\documentclass{amsart}

\usepackage{mystyle}

\title{Regular sequences for triangulated categories}
\date{\today}

\keywords{regular elements, regular sequences, triangulated categories, derived categories, ghost maps, local cohomology, Koszul objects, level, Rouquier dimension}
\subjclass[2020]{18G80 (primary), 13D09, 13D45}

\author[A.~Kekkou]{Antonia Kekkou}
\address{Antonia Kekkou, 
Department of Mathematics,
University of Utah, 
UT 84112, U.S.A.}
\email{kekkou@math.utah.edu}

\author[J.~C.~Letz]{Janina C. Letz}
\address{Janina~C.~Letz,
Faculty of Mathematics,
Bielefeld University,
PO Box 100 131,
33501 Bielefeld,
Germany}
\email{jletz@math.uni-bielefeld.de}

\author[M.~Stephan]{Marc Stephan}
\address{Marc Stephan,
Faculty of Mathematics,
Bielefeld University,
PO Box 100 131,
33501 Bielefeld,
Germany}
\email{marc.stephan@math.uni-bielefeld.de}

\begin{document}

\begin{abstract}
This paper systematically develops a notion of regular sequences in the context of $R$-linear triangulated categories for a graded-commutative ring $R$.
The notion has equivalent characterizations involving Koszul objects and local cohomology.
The main examples are in the context of the Hochschild cohomology ring or the group cohomology ring acting on derived or stable categories. 
As applications, lengths of regular sequences provide lower bounds for level and Rouquier dimension. 
\end{abstract}

\maketitle
\setcounter{tocdepth}{1}
\tableofcontents

%%%%%%%%%%%%%%%%%%%%%%%%%%%%%%%%%%%%%%%%%%%%%%%%%%
%%%%%%%%%%%%%%%%%%%%%%%%%%%%%%%%%%%%%%%%%%%%%%%%%%
\section{Introduction}

Regular sequences play a central role in commutative algebra as they have useful homological properties. We systematically develop a notion of regular sequences in the context of $R$-linear triangulated categories for a graded-commutative ring $R$. Our notion of regular sequences is compatible with the triangulated structure, namely suspension and exact triangles. 

A motivating example is the derived category of commutative ring where our definition extends the classical definition for regular sequences of modules. We start from the characterization of regular sequences on a module via vanishing of Koszul homology and local cohomology. The Koszul complex of a complex has two canonical filtrations, the filtration by the homological degree and the filtration obtained from its construction. Hence vanishing of Koszul homology can be generalized with either of these filtrations. Previous definitions of regular sequences for complexes of modules by \cite{Christensen:2001,Minamoto:2021} used the former filtration, we use the latter, as it respects the triangulated structure.

In a triangulated category, Koszul objects generalize Koszul complexes. The Koszul object $\kosobj{M}{x}$ of $x \in R$ on the object $M$ is equipped with a morphism $\kosobj{M}{x} \to \susp M$ corresponding to the truncation map to degree $\geq 1$ for a Koszul complex. The local cohomology functor $\Gamma_{\cV(x)}$ for a triangulated category as introduced in \cite{Benson/Iyengar/Krause:2008} generalizes the derived torsion functor whose homology is local cohomology. It is equipped with a natural transformation $\Gamma_{\cV(x)} M \to M$. 

\begin{introthm} \label{intro:weakly_regular}
Let $R$ be a graded-commutative ring and $\cat{T}$ an $R$-linear triangulated category. Let $C,M \in \cat{T}$ and $x_1, \ldots, x_t$ a sequence in $R$. The following are equivalent
\begin{enumerate}
\item $x_1, \ldots, x_t$ is weakly $\Hom[*]{\cat{T}}{C}{M}$-regular in the classical sense; 
\item each morphism in the composition
\begin{equation*}
\susp^{-t} \kosobj{M}{(x_1, \ldots, x_t)} \to \susp^{-t+1} \kosobj{M}{(x_1, \ldots, x_{t-1})} \to \dots \to \susp^{-1} \kosobj{M}{x_1} \to M
\end{equation*}
is $C$-ghost. 
\end{enumerate}
If the ring $R$ is noetherian, the triangulated category $\cat{T}$ is compactly generated and $C$ is compact, then the above conditions are equivalent to
\begin{enumerate}[resume]
\item each morphism in the composition
\begin{equation*}
\Gamma_{\cV(x_1, \ldots, x_t)} M \to \Gamma_{\cV(x_1, \ldots, x_{t-1})} M\to \dots \to \Gamma_{\cV(x_1)} M \to M
\end{equation*}
is $C$-ghost.
\end{enumerate}
\end{introthm}

This \namecref{intro:weakly_regular} follows from \cref{regular_nzdivisor,reg_seq_equiv}. We say a sequence $x_1, \ldots, x_t$ is \emph{$(C,M)$-regular} if it satisfies the conditions from \cref{intro:weakly_regular} and the composition $\susp^{-t} \kosobj{M}{(x_1, \ldots, x_t)} \to M$ is non-zero. There is no characterization of the latter condition in terms of the local cohomology functor in general. We provide a characterization under additional assumptions in \cref{subsec:nonzero}. %When $C = M$, an $(M,M)$-regular sequence is exactly a regular sequence for the graded $R$-module $\Gamma_{\cV(x)} M \to M$ in the classical sense.

Our definition of regular sequences recovers the classical definition for a module $M$ over a commutative ring $A$ by taking $\cat{T}=\dcat{A}$, the derived category of a commutative ring, $R=A$ and $C=A$. 

While our notion of regular sequences is new, examples have appeared in various contexts. For a field $k$ and a finite group $G$, the derived category $\cat{T}=\dcat{kG}$ is linear over the group cohomology $\Ext[*]{kG}{k}{k}$. The $(k,k)$-regular sequences are the classical regular sequences in the group cohomology ring.

When $A$ is a commutative ring and $B$ a dg algebra over $A$, then the Hochschild--Shukla cohomology $\HH{B/A}$ acts on the derived category $\dcat{B}$. When $A$ is a regular local ring with residue field $k$, and $B$ is a Koszul complex over $A$ or a complete intersection, then there is a canonical $(k,k)$-regular sequence. Utilizing this regular sequence, we can bypass proofs using BGG correspondence to recover, and generalize, several results by \cite{Avramov/Buchweitz/Iyengar/Miller:2010,Briggs/Grifo/Pollitz:2024} from the following result.

\begin{introthm} \label{main_level_lower_bound_reg}
Let $R$ be a graded-commutative ring and $\cat{T}$ an $R$-linear triangulated category. Let $C,M \in \cat{T}$ and $x_1, \ldots, x_t \in R$ a $(C,C)$-regular sequence. If $M \in \thick_\cat{T}(\kosobj{C}{(x_1, \ldots, x_t)})$ and $M\neq 0$, then
\begin{equation*}
\level_\cat{T}^C(M) \geq t + 1\,.
\end{equation*}
\end{introthm}

This result is proven in \cref{level_lower_bound_reg}, and its connection to \cite{Avramov/Buchweitz/Iyengar/Miller:2010,Briggs/Grifo/Pollitz:2024} via Hochschild cohomology is explained in-depth in \cref{level_HH_recover}. In particular, we recover precursors from equivariant topology, the rank inequalities for free actions of elementary abelian $p$-groups on finite CW-complexes from \cite{Carlsson:1983,Baumgartner:1993,Allday/Puppe:1993}.

As another application, due to the second author, the length of any $(M,M)$-regular sequence provides a lower bound to Rouquier dimension whenever the triangulated category is Ext-finite; see \cite[Theorem~B]{Letz:2025}. In particular, the Rouquier dimension of the bounded derived category of finitely generated modules over a commutative noetherian ring $A$ is bounded below by the depth of the ring $A$. When $A$ is a complete intersection ring, we show that the Rouquier dimension of the bounded derived category of a complete intersection is also bounded below by the codimension of $A$. For a finite group algebra $kG$ over a field $k$ of characteristic $p$, we show that the Rouquier dimension of the bounded derived category of $kG$ is bounded below by the $p$-rank of $G$. This improves the bounds from \cite[Corollary~5.10]{Bergh/Iyengar/Krause/Oppermann:2010} and \cite{Oppermann:2007,Bergh/Iyengar/Krause/Oppermann:2010}, respectively, by one.

To prove \cref{intro:weakly_regular} we connect Koszul objects and local cohomology. In commutative algebra they are connected via the \v{C}ech complex, which was generalized to the triangulated setting for one element by \cite{Benson/Iyengar/Krause:2011b}. We generalize this to a sequence of elements.

\begin{introthm}\label{introthm:kos_local_coh}
Let $R$ be a graded-commutative noetherian ring and $\cat{T}$ a compactly generated $R$-linear triangulated category. For any sequence of elements $x_1, \ldots, x_t$ in $R$, there exists an isomorphism
\begin{equation*}
\hocolim_n \susp^{-t} \kosobj{M}{(x_1^n, \dots, x_t^n)} \xrightarrow{\cong} \Gamma_{\cV(x_1, \ldots, x_t)} M \nospacepunct{.}
\end{equation*}
\end{introthm}

We prove this result in \cref{Gamma_hocolim}, where we also show that the isomorphism is compatible with the canonical morphisms to $M$.

\begin{ack}
We thank Srikanth Iyengar and Henning Krause for helpful discussions.

The second and third authors were partly funded by the Deutsche Forschungsgemeinschaft (DFG, German Research Foundation)---Project-ID 491392403---TRR 358. The first author was partly supported by National Science Foundation grant DMS-200985.
\end{ack}

%%%%%%%%%%%%%%%%%%%%%%%%%%%%%%%%%%%%%%%%%%%%%%%%%%
%%%%%%%%%%%%%%%%%%%%%%%%%%%%%%%%%%%%%%%%%%%%%%%%%%
\section{Koszul objects}

Let $R$ be a graded-commutative ring. This means $R = \bigoplus_{n \in \BZ} R_n$ is $\BZ$-graded and $xy = (-1)^{|x| |y|}yx$ where $|x|$ and $|y|$ are degrees of $x$ and $y$, respectively. All elements and ideals will be homogeneous in a graded ring.

%%%%%%%%%%%%%%%%%%%%%%%%%%%%%%%%%%%%%%%%%%%%%%%%%%
\subsection{Linear triangulated categories}

A triangulated category $\cat{T}$ is \emph{$R$-linear}, if it is equipped with a homomorphism
$R \to \ctr{\cat{T}}$ of graded rings, where 
\begin{equation*}
 \ctr{\cat{T}} = \bigoplus_{n\in \BZ} \set{ \alpha\colon \id_{\cat{T}} \to \susp^n | \alpha(\susp) = (-1)^n \susp \alpha}
\end{equation*}
 is the graded center of $\cat{T}$; see \cite[Section~3]{Buchweitz/Flenner:2008}. It follows that
\begin{equation*}
\Hom[*]{\cat{T}}{M}{N} \colonequals \bigoplus_{n \in \BZ} \Hom {\cat{T}}{M}{\susp^n N}
\end{equation*}
is a graded $R$-module for any objects $M$ and $N$.
Often, we write
\begin{equation*}
x = x(M) = x \cdot \id_M \colon M \to \susp^{|x|} M
\end{equation*}
for the induced morphism of $x \in R$ on an object $M \in \cat{T}$, as $x$ is natural in $M$. 

\begin{remark}\label{rem:homological_grading}
Just as one may prefer to work with graded endomorphism rings $\End[\cat{T}]{*}{M}=\bigoplus_{n\in\BZ} \Hom[]{\cat{T}}{\susp^n M}{M}$
instead of $\End[*]{\cat{T}}{M}=\Hom[*]{\cat{T}}{M}{M}$, the definition
\begin{equation*}
 \operatorname{Z}_{*}(\cat{T}) = \bigoplus_{n\in \BZ} \set{ \alpha\colon \susp^n\to \id_{\cat{T}} | \alpha(\susp ) = (-1)^n \susp \alpha}
\end{equation*}
may be more suitable in certain homologically graded contexts. We leave it to the reader to adjust the definitions and results. Note that $\operatorname{Z}_{-*}(\cat{T}) \cong \ctr{\cat{T}}$.
\end{remark}

%%%%%%%%%%%%%%%%%%%%%%%%%%%%%%%%%%%%%%%%%%%%%%%%%%
\subsection{Koszul objects}

Given a sequence $x_1, \ldots, x_t \in R$ we define the \emph{Koszul object} of $M$ with respect to $\bmx = x_1, \ldots, x_t$ as
\begin{equation*}
\kosobj{M}{\bmx} = \kosobj{M}{(x_1, \ldots, x_t)} \colonequals \begin{cases} 
M & t = 0 \\ 
\cone(M \xrightarrow{x} \susp^{|x|} M) & t = 1 \\
\kosobj{(\kosobj{M}{(x_1, \ldots, x_{t-1})})}{x_t} & t \geq 2\,.
\end{cases}
\end{equation*}
By construction, the Koszul object $\kosobj{M}{x}$ comes with an exact triangle
\begin{equation} \label{kosobj_triangle}
M \xrightarrow{x} \susp^{|x|} M \to \kosobj{M}{x} \to \susp M\,.
\end{equation}
We call this the \emph{defining exact triangle} of $\kosobj{M}{x}$.  

%%%%%%%%%%%%%%%%%%%%%%%%%%%%%%%%%%%%%%%%%%%%%%%%%%
\subsection{Koszul object of a product} \label{kos_product}

Let $x,y \in R$. Applying the octahedral axiom to $xy$ on an object $M$ yields the exact triangle
\begin{equation} \label{kos_product_triangle}
\kosobj{M}{y} \to \kosobj{M}{xy} \to \susp^{|y|} \kosobj{M}{x} \to \susp \kosobj{M}{y}
\end{equation}
which is compatible with the defining exact triangles of $\kosobj{M}{x}$, $\kosobj{M}{y}$ and $\kosobj{M}{xy}$ in the sense that the following diagram commutes
\begin{equation} \label{kos_product_diagram}
\begin{tikzcd}[column sep=tiny]
M \ar[rr,bend left,"xy"] \ar[dr,"y"] &[+2em]& \susp^{|x|+|y|} M \ar[rr,bend left] \ar[dr] && \susp^{|y|} \kosobj{M}{x} \ar[rr,bend left] \ar[dr] && \susp \kosobj{M}{y} \nospacepunct{.} \\
& \susp^{|y|} M \ar[ur,"x"] \ar[dr] && \kosobj{M}{xy} \ar[ur,dashed] \ar[dr] && \susp^{|y|+1} M \ar[ur] \\
&& \kosobj{M}{y} \ar[ur,dashed] \ar[rr,bend right] && \susp M \ar[ur,"y"]
\end{tikzcd}
\end{equation}

%%%%%%%%%%%%%%%%%%%%%%%%%%%%%%%%%%%%%%%%%%%%%%%%%%
\subsection{Linear functors}

An exact functor $\sF \colon \cat{S} \to \cat{T}$ of $R$-linear triangulated categories is \emph{$R$-linear}, if the induced map
\begin{equation*}
\Hom[*]{\cat{S}}{M}{N} \to \Hom[*]{\cat{T}}{\sF(M)}{\sF(N)}
\end{equation*}
is a morphism of graded $R$-modules for any $M,N \in \cat{S}$. In this situation one has $\sF(\kosobj{M}{x}) \cong \kosobj{\sF(M)}{x}$ for $x \in R$ as $\sF(x(M))$ is the composite of $x(\sF(M))$ and $\sF(\susp^{|x|}M) \cong \susp^{|x|}\sF(M)$. For example, the suspension functor $\sF = \susp$ together with the natural isomorphism $(-\id) \colon \sF \susp\to \susp \sF$ is exact and $R$-linear. 

%%%%%%%%%%%%%%%%%%%%%%%%%%%%%%%%%%%%%%%%%%%%%%%%%%
\subsection{Functoriality} \label{kos_functorial}

The Koszul object is unique up to non-unique isomorphism. In particular, it is not functorial. Explicitly, a morphism $M \to N$ in $\cat{T}$ induces a morphism $\kosobj{M}{x} \to \kosobj{N}{x}$ that fits into a morphism of exact triangles
\begin{equation} \label{kosobj_induced_map}
\begin{tikzcd}
M \ar[r,"x"] \ar[d] & \susp^{|x|} M \ar[r] \ar[d] & \kosobj{M}{x} \ar[r] \ar[d,dashed] & \susp M \ar[d] \\
N \ar[r,"x"] & \susp^{|x|} N \ar[r] & \kosobj{N}{x} \ar[r] & \susp N \nospacepunct{.}
\end{tikzcd}
\end{equation} 
Yet, the induced morphism is not unique, and it might not be possible to extend the morphism of exact triangle to a 3$\times$3-diagram in which all rows and columns are exact triangles. 

However, there are classes of elements whose Koszul object is functorial. One such class of elements is induced by the action of a tensor triangulated category; see \cite[4.6]{Letz/Stephan:2025}. For an example of a Koszul object that is not functorial, see \cite[Example~4.7]{Letz/Stephan:2025}. 

\begin{definition}
We say an element $x \in R$ is \emph{Koszul-exact} if the assignment $M \mapsto \kosobj{M}{x}$ on objects can be made into an $R$-linear exact functor such that
\begin{equation*}
\id \xrightarrow{x} \susp^{|x|} \to \kosobj{-}{x} \to \susp
\end{equation*}
is a functorial exact triangle. This means, for every object $M$ there is an exact triangle \cref{kosobj_triangle} and the maps are natural transformations of exact functors.

In particular, the morphism of exact triangles \cref{kosobj_induced_map} can be canonically extended to a 3$\times$3-diagram in which all columns and rows are exact triangles.
\end{definition}

%%%%%%%%%%%%%%%%%%%%%%%%%%%%%%%%%%%%%%%%%%%%%%%%%%
\subsection{Commutativity} \label{kos_obj_commute}

Let $x,y \in R$. If $x$ or $y$ is Koszul-exact, then there exists an isomorphism
\begin{equation*}
\kosobj{M}{(x,y)} \to \kosobj{M}{(y,x)}
\end{equation*}
for any $M \in \cat{T}$ that is compatible with the canonical morphisms of the Koszul object; that is the following diagrams (anti)commute
\begin{equation} \label{Kos_functor_compatible}
\begin{tikzcd}[column sep=0]
\susp^{-2} \kosobj{M}{(x,y)} \ar[rr,"\cong"] \ar[d] && \susp^{-2} \kosobj{M}{(y,x)} \ar[d] \\
\susp^{-1} \kosobj{M}{x} \ar[dr] && \susp^{-1} \kosobj{M}{y} \ar[dl] \\
& M
\end{tikzcd}
\quad \text{and} \quad
\begin{tikzcd}[column sep=-1em]
& \susp^{|x|+|y|} M \ar[dl] \ar[dr] \ar[dd,phantom,"(-1)",pos=0.6] \\
\susp^{|y|} \kosobj{M}{x} \ar[d] && \susp^{|x|} \kosobj{M}{y} \ar[d] \\
\kosobj{M}{(x,y)} \ar[rr,"\cong"] && \kosobj{M}{(y,x)} \nospacepunct{.}
\end{tikzcd}
\end{equation}

%%%%%%%%%%%%%%%%%%%%%%%%%%%%%%%%%%%%%%%%%%%%%%%%%%
\subsection{Opposite category}

Let $\cat{T}$ be a triangulated category. Then its opposite category $\op{\cat{T}}$ is a triangulated category with suspension $\susp^{-1}$ and exact triangles $(\op{f},\op{g},\op{h})$ whenever $(-h,-g,-f)$ is exact in $\cat{T}$. 

We assume $\cat{T}$ is $R$-linear for a graded-commutative ring $R$. Set
\begin{equation*}
\op{x}(M) \colonequals \susp^{-|x|} \op{(x(M))} \colon M \to \susp^{-|x|} M\,.
\end{equation*}
This defines an element in $\ctr{\op{\cat{T}}}$. As $\op{(xy)}(M) = (\op{y} \op{x})(M)$, the opposite category $\op{\cat{T}}$ is $\op{R}$-linear. Moreover, this yields an isomorphism
\begin{equation*}
\op{\ctr{\cat{T}}} \cong \ctr{\op{\cat{T}}}
\end{equation*} 
of the graded centers of a triangulated category and its opposite triangulated category.

%%%%%%%%%%%%%%%%%%%%%%%%%%%%%%%%%%%%%%%%%%%%%%%%%%
\subsection{Productive elements}

Let $R$ be a graded-commutative ring acting on a triangulated category $\cat T$. We are interested in elements $x\in R$ that vanish on their Koszul objects, that is, such that $x(\kosobj{M}{x})=0$ for any $M\in \cat T$. This does not hold in general; see, for example, \cite[Proposition~4]{Schwede:2010}. 

We say an element $x \in R$ is \emph{$M$-productive} for an object $M$ on $\cat{T}$ if $x(\kosobj{M}{x}) = 0$. If $x$ is $M$-productive for all objects $M$, then $x$ is called \emph{productive}. This notation is taken from the corresponding property in group cohomology; compare \cite{Carlson:1987,Carlson:1996}.

\begin{lemma} \label{productive_product}
Let $R$ be a graded-commutative ring acting on a triangulated category $\cat T$. If $x,y \in R$ are productive, then so is $xy$. 
\end{lemma}
\begin{proof}
Apply $xy$ to the exact triangle \cref{kos_product_triangle}. The claim now follows from a standard diagram chase.
\end{proof}

\begin{remark}
For a natural number $n$ and any triangulated category $\cat T$ considered as a $\BZ$-linear category via the canonical morphism $\BZ\to \ctr{\cat T}$, the vanishing of $n\colon \kosobj{M}{n}\to \kosobj{M}{n}$ can be expressed via the $n$-order introduced in \cite{Schwede:2010, Schwede:2013}. The morphism $n$ on $\kosobj{M}{n}$ is zero if and only if the $n$-order of $\kosobj{M}{n}$ is at least one. It is natural to extend the notion of order to any $R$-linear triangulated category to study the vanishing of $x(\kosobj{M}{x})$ for any $x\in R$. This notion of $x$-order appears in \cite{Langer:2009}.
\end{remark}

%%%%%%%%%%%%%%%%%%%%%%%%%%%%%%%%%%%%%%%%%%%%%%%%%%
%%%%%%%%%%%%%%%%%%%%%%%%%%%%%%%%%%%%%%%%%%%%%%%%%%
\section{The local cohomology functor} \label{sec:local_coh}

We first recall the notions of the Zariski spectrum and of compact objects in triangulated categories, as they are used in the definition of the local cohomology functor. Benson, Iyengar and Krause connected the local cohomology functor associated to a principal ideal $(x)$ to the Koszul object with respect to $x$ in \cite[Proposition~2.9]{Benson/Iyengar/Krause:2011b}. We extend this connection to a sequence of elements. 

%%%%%%%%%%%%%%%%%%%%%%%%%%%%%%%%%%%%%%%%%%%%%%%%%%
\subsection{Zariski spectrum}

The Zariski spectrum of a graded-commutative ring $R$, denoted by $\Spec(R)$, is the set of homogeneous prime ideals of $R$ with the topology given by the closed sets
\begin{equation*}
\cV(\fa) \colonequals \Set{\fp \in \Spec(R) | \fa \subseteq \fp}
\end{equation*}
for any homogeneous ideal $\fa \subseteq R$. 

A subset $\cV \subseteq \Spec(R)$ is \emph{specialization closed} if $\fp \in \cV$ and $\fp \subseteq \fq$ imply $\fq \in \cV$. In fact, a set is specialization closed if and only if it is the union of closed sets.

%%%%%%%%%%%%%%%%%%%%%%%%%%%%%%%%%%%%%%%%%%%%%%%%%%
\subsection{Compact objects}

An object $C$ in a triangulated category $\cat{T}$ is called \emph{compact}, if $\Hom{\cat T}{C}{-}$ preserves small coproducts. The full subcategory of compact objects $\compact{\cat{T}}$ is triangulated and closed under direct summands. A triangulated category is \emph{compactly generated}, if it has all small coproducts and there is a set of compact objects $\cat{C}$ such that $\cat{T} = \loc(\cat{C})$; that is $\cat{T}$ is the smallest triangulated subcategory that is closed under small coproducts and contains $\cat{C}$. For a compactly generated triangulated category $\cat{T}$, the subcategory of compact objects $\compact{\cat{T}}$ is essentially small.

For a compact object $C$, one can view $\ch{C}{-} \colonequals \Hom[*]{\cat{T}}{C}{-}$ as a homology theory; this means it is a homological functor and preserves small coproducts. In a compactly generated triangulated category, an object $M$ is zero if and only if $\ch{C}{M} = 0$ for all compact objects $C$; see for example \cite[Lemma~2.2.1]{Schwede/Shipley:2003}. 

%%%%%%%%%%%%%%%%%%%%%%%%%%%%%%%%%%%%%%%%%%%%%%%%%%
\subsection{Local cohomology functor}

We recall the construction of the local cohomology functor of Benson, Iyengar, and Krause from \cite{Benson/Iyengar/Krause:2008,Benson/Iyengar/Krause:2011b}.

Let $R$ be a graded-commutative noetherian ring and $\cat{T}$ a compactly generated $R$-linear triangulated category. For each specialization closed subset $\cV$ of $\Spec(R)$, the full subcategory of $\cat{T}$ of \emph{$\cV$-torsion} objects is
\begin{equation*}
\cat{T}_\cV \colonequals \loc(\Set{C \in \compact{\cat{T}} | \End[*]{\cat{T}}{C}_\fp = 0 \text{ for all } \fp \notin \cV})\,;
\end{equation*}
this agrees with the definition from \cite[Section~4]{Benson/Iyengar/Krause:2008} by \cite[Theorem~6.4]{Benson/Iyengar/Krause:2008}. The category $\cat{T}_\cV$ is a compactly generated localizing subcategory of $\cat{T}$ and the inclusion functor $\cat{T}_\cV \to \cat{T}$ has a right adjoint by \cite[5.3]{Krause:2010}. Then Bousefield localization yields a localization functor $\rL_\cV \colon \cat{T} \to \cat{T}$ with $\ker(\rL_\cV) = \cat{T}_\cV$ and a colocalization functor $\Gamma_{\cV} \colon \cat{T} \to \cat{T}$ with essential image $\operatorname{im}(\Gamma_{\cV}) = \cat{T}_{\cV}$ and there is a functorial exact triangle
\begin{equation*}
\Gamma_{\cV} \to \id \to \rL_{\cV} \to \susp \Gamma_{\cV}\,;
\end{equation*}
see \cite[Section~4.11]{Krause:2010}. The functor $\Gamma_{\cV}$ is called the \emph{local cohomology functor} associated to $\cV$. 

When $\cV = \cV(x_1, \ldots, x_t)$ for $x_1, \ldots, x_t \in R$, then
\begin{equation*}
\cat{T}_\cV = \loc(\Set{\kosobj{C}{(x_1, \ldots, x_t)} | C \in \compact{\cat{T}}})
\end{equation*}
as the compact generating sets coincide by \cite[Lemma~3.9, Proposition~3.10]{Benson/Iyengar/Krause:2015}.

%%%%%%%%%%%%%%%%%%%%%%%%%%%%%%%%%%%%%%%%%%%%%%%%%%
\subsection{Support} \label{support}

For a homogeneous prime ideal $\fp$ of $R$, let
\[
\cZ(\fp) \colonequals \set{\fq \in \Spec(R) | \fq \nsubseteq \fp}
\]
and $\Gamma_\fp \colonequals \Gamma_{\cV(\fp)}\rL_{\cZ(\fp)}$. The \emph{support} of an object $M$ of $\cat{T}$ is defined as
\[
\supp_R(M) \colonequals \set{\fp\in \Spec(R) | \Gamma_\fp M \neq 0}\,.
\]
If $\cat{T}$ has a compact generator $C$ such that $\ch{C}{M}$ is a finitely generated graded $R$-module, then
\[
\supp_R(M) = \set{\fp \in \Spec(R) | \ch{C}{M}_\fp\neq 0} = \cV(\ann_R \ch{C}{M})
\,,
\]
where $\ann_R N$ denotes the annihilator ideal of a graded $R$-module $N$; see \cite[Lemma~2.2, Theorem~5.5(1)]{Benson/Iyengar/Krause:2008}.

%%%%%%%%%%%%%%%%%%%%%%%%%%%%%%%%%%%%%%%%%%%%%%%%%%
\subsection{Linear functors} \label{linear_functors}

By construction, the local cohomology functor is the precomposition of the inclusion $\cat{T}_\cV \to \cat{T}$ with its right adjoint. As the inclusion is $R$-linear, the local cohomology functor is $R$-linear by the following \namecref{adj_linear}.

\begin{lemma} \label{adj_linear}
Let $\sF \colon \cat{C} \to \cat{D}$ be a functor between $R$-linear triangulated categories with a right adjoint $\sG \colon \cat{D} \to \cat{C}$. Then $\sF$ is $R$-linear if and only if $\sG$ is $R$-linear. 
\end{lemma}
\begin{proof}
We show the forwards direction, the backwards direction holds by the same argument in $\op{\cat{T}}$. 
We write $\epsilon \colon \sF\sG\to \id$ for the counit of the adjunction. The adjunction isomorphism
\[
\Hom{\cat{C}}{M}{\sG(N)}\cong \Hom{\cat{D}}{\sF(M)}{N} \,, \quad f \mapsto \epsilon(N) \sF(f)
\]
induces an isomorphism
\[
\Hom[*]{\cat{C}}{M}{\sG(N)}\cong \Hom[*]{\cat{D}}{\sF(M)}{N}\,.
\]
As $x \in R$ is natural, it commutes with $\epsilon(N)$, and as $\sF$ is $R$-linear, $x$ also commutes with $\sF$. Hence the isomorphism is $R$-linear.

It follows that the morphism $\Hom[*]{\cat{D}}{L}{N} \to \Hom[*]{\cat{C}}{\sG(L)}{\sG(N)}$ induced by $\sG$ is $R$-linear since it is the composite of $R$-linear maps
\[
\Hom[*]{\cat{D}}{L}{N}\to \Hom[*]{\cat{D}}{\sF\sG(L)}{N}\cong \Hom[*]{\cat{C}}{\sG(L)}{\sG(N)}\,. \qedhere
\]
\end{proof}

\begin{remark}
The Brown representability theorem for compactly generated triangulated categories by Neeman \cite{Neeman:1996} states that any cohomological functor $\sH \colon \cat{T} \to \op{\abcat}$ that preserves small coproducts is naturally isomorphic, as abelian groups, to $\Hom{\cat{T}}{-}{M}$ for some $M \in \cat{T}$. Slightly adjusting the proof should show that, when $\cat{T}$ is $R$-linear and $\sH \colon \cat{T} \to \op{(\Mod{R})}$ a cohomological functor, then the natural isomorphism of $\sH$ and $\Hom{\cat{T}}{-}{M}$ is $R$-linear. Since the local cohomology functor is constructed using Brown representability, this would yield an alternative proof that the local cohomology functor is $R$-linear. 
\end{remark}

%%%%%%%%%%%%%%%%%%%%%%%%%%%%%%%%%%%%%%%%%%%%%%%%%%
\subsection{Homotopy colimits}

Let $M_1 \xrightarrow{f_1} M_2 \xrightarrow{f_2} M_3 \to \cdots$ be a sequence of morphisms in a triangulated category $\cat{T}$ with small coproducts. The \emph{homotopy colimit} $\hocolim_n M_n$ of the sequence is defined via the exact triangle
\begin{equation*}
\coprod_n M_n \xrightarrow{\id-f_n} \coprod_n M_n \to \hocolim_n M_n \to \susp \coprod_n M_n\,.
\end{equation*}
The homotopy colimit is equipped with canonical morphisms
\begin{equation} \label{hpty_colimit_can_map}
M_\ell \to \coprod_n M_n \to \hocolim_n M_n
\end{equation}
for any $\ell \geq 1$. Given morphisms $M_\ell \to N$ for every $\ell \geq 1$ compatible with the $f_\ell$'s, there exists a non-unique morphism $\hocolim_n M_n \to N$ such that the following diagram commutes
\begin{equation*}
\begin{tikzcd}
M_\ell \ar[dr] \ar[d] \\
\hocolim_n M_n \ar[r] & N
\end{tikzcd}
\end{equation*}
for every $\ell \geq 1$.

While the homotopy colimit is not a colimit, for any compact object $C$ in $\cat{T}$ one has
\begin{equation} \label{hocolim_colim_compact}
\ch{C}{\hocolim_n M_n} \cong \colim_n \ch{C}{M_n}\,;
\end{equation}
see \cite[Lemma~3.4.3]{Krause:2022}.

%%%%%%%%%%%%%%%%%%%%%%%%%%%%%%%%%%%%%%%%%%%%%%%%%%
\subsection{Connection between local cohomology and Koszul objects} \label{Kos_obj_Gamma}

Let $R$ be a graded-commutative noetherian ring and $\cat{T}$ a compactly generated $R$-linear triangulated category. By \cite[Proposition~2.9]{Benson/Iyengar/Krause:2011b}, there exist isomorphisms
\begin{equation} \label{Kos_obj_Gamma_one}
\begin{gathered}
\hocolim (M \xrightarrow{x} \susp^{|x|} M \xrightarrow{x} \susp^{2|x|} M \to \ldots) \xrightarrow{\cong} \rL_{\cV(x)} M \quad \text{and} \\
\quad \susp^{-1} \hocolim (\kosobj{M}{x} \to \kosobj{M}{x^2} \to \ldots) \xrightarrow{\cong} \Gamma_{\cV(x)} M
\end{gathered}
\end{equation}
for any $x \in R$ and $M \in \cat{T}$.

In the following we generalize the isomorphisms \cref{Kos_obj_Gamma_one} to Koszul objects with respect to a sequence of elements $x_1, \ldots, x_t \in R$. As the homotopy colimit is not a colimit, this cannot be simply deduced from the isomorphisms for one element. 

Let $x_1, \ldots, x_t$ be a sequence of elements in $R$ for $t \geq 1$. For convenience we write $\bmx_s = x_1, \ldots, x_s$ and $\bmx_s^n = x_1^n, \ldots, x_s^n$ for any $n \geq 1$ and $0 \leq s \leq t$. First we construct a sequence of morphisms of exact triangles. We inductively define morphisms $u_{s,n} \colon \susp^{-s} \kosobj{M}{\bmx_s^n} \to \susp^{-s} \kosobj{M}{\bmx_s^{n+1}}$ for $0 \leq s \leq t$. We set $u_{0,n} = \id_M$. Given $u_{s-1,n}$, we let $u_{s,n}$ be a morphism that completes the commutative diagram
\begin{equation} \label{Kos_obj_Gamma:mor_u}
\begin{tikzcd}[column sep=small]
\susp^{-s} \kosobj{M}{\bmx_s^n} \ar[r] \ar[d,"{u_{s,n}}",dashed] &[-.2em] \susp^{-s+1} \kosobj{M}{\bmx_{s-1}^n} \ar[r,"{x_s^n}"] \ar[d,"{u_{s-1,n}}"] & \susp^{n|x_s|-s+1} \kosobj{M}{\bmx_{s-1}^n} \ar[r] \ar[d,"{x_s \susp^{|x_s|} u_{s-1,n}}"] &[-.2em] \susp^{-s+1} \kosobj{M}{\bmx_s^n} \ar[d,"{\susp u_{s,n}}",dashed] \\
\susp^{-s} \kosobj{M}{\bmx_s^{n+1}} \ar[r] & \susp^{-s+1} \kosobj{M}{\bmx_{s-1}^{n+1}} \ar[r,"{x_s^{n+1}}"] & \susp^{(n+1)|x_s|-s+1} \kosobj{M}{\bmx_{s-1}^{n+1}} \ar[r] & \susp^{-s+1} \kosobj{M}{\bmx_s^{n+1}}
\end{tikzcd}
\end{equation}
to a morphism of exact triangles. We let $V_{t,n}$ be the cone of the composition
\begin{equation} \label{Kos_obj_Gamma:composition_V}
\susp^{-t} \kosobj{M}{(x_1^n, \ldots, x_t^n)} \to \susp^{-(t-1)} \kosobj{M}{(x_1^n, \ldots, x_{t-1}^n)} \to \cdots \to \susp^{-1} \kosobj{M}{x_1^n} \to M
\end{equation}
and let $V_{t,n} \to V_{t,n+1}$ be a morphism that completes the commutative diagram
\begin{equation} \label{Kos_obj_Gamma:mor_v}
\begin{tikzcd}
\susp^{-t} \kosobj{M}{\bmx_t^n} \ar[r] \ar[d,"u_{t,n}"] & M \ar[r] \ar[d,"="] & V_{t,n} \ar[r] \ar[d,"{v_{t,n}}",dashed] & \susp^{-t+1} \kosobj{M}{\bmx_t^n} \ar[d,"{\susp u_{t,n}}"] \\
\susp^{-t} \kosobj{M}{\bmx_t^{n+1}} \ar[r] & M \ar[r] & V_{t,n+1} \ar[r] & \susp^{-t+1} \kosobj{M}{\bmx_t^{n+1}}
\end{tikzcd}
\end{equation}
to a morphism of exact triangles. 

\begin{theorem} \label{Gamma_hocolim}
Let $R$ be a graded-commutative noetherian ring and $\cat{T}$ a compactly generated $R$-linear triangulated category. For any sequence of elements $x_1, \ldots, x_t$ in $R$, there exist isomorphisms
\begin{equation*}
\begin{gathered}
\hocolim_n V_{t,n} \xrightarrow{\cong} \rL_{\cV(x_1, \ldots, x_t)} M \quad \text{and} \\
\hocolim_n \susp^{-t} \kosobj{M}{(x_1^n, \ldots, x_t^n)} \xrightarrow{\cong} \Gamma_{\cV(x_1, \ldots, x_t)} M
\end{gathered}
\end{equation*}
making
\begin{equation} \label{Kos_obj_Gamma_can_mor}
\begin{tikzcd}
\susp^{-t} \kosobj{M}{(x_1^\ell, \ldots, x_t^\ell)} \ar[r] \ar[d] & M \ar[dd,"="] \ar[r] & V_{t,\ell} \ar[d] \\
\hocolim_n \susp^{-t} \kosobj{M}{(x_1^n, \ldots, x_t^n)}\ar[d,"\cong"] & & \hocolim_n V_{t,n}\ar[d,"\cong"] \\
\Gamma_{\cV(x_1, \ldots, x_t)} M \ar[r] & M \ar[r] & \rL_{\cV(x_1, \ldots, x_t)} M 
\end{tikzcd}
\end{equation}
commute for any $\ell \geq 1$.
\end{theorem}

\begin{remark}
It is not clear whether \eqref{Kos_obj_Gamma_can_mor} is a morphism of exact triangles, that is whether 
\begin{equation*}
\begin{tikzcd}
V_{t,\ell} \ar[r] \ar[d] & \susp^{1-t} \kosobj{M}{(x_1^\ell, \ldots, x_t^\ell)} \ar[d] \\
\rL_{\cV(x_1, \ldots, x_t)} M \ar[r] & \susp \Gamma_{\cV(x_1, \ldots, x_t)} M
\end{tikzcd}
\end{equation*}
commutes.
\end{remark}

\begin{proof}
We proceed in four steps.
\begin{step} \label{Gamma_hocolim:HH_Gamma_Kos}
For any compact object $C$, there is an induced isomorphism
\begin{equation*}
\colim_n \ch{C}{\Gamma_{\cV(\bmx_t)} \susp^{-t} \kosobj{M}{\bmx_t^n}} \cong \ch{C}{\Gamma_{\cV(\bmx_t)} M}\,.
\end{equation*}
\end{step}

We establish \cref{Gamma_hocolim:HH_Gamma_Kos} using induction on $t$. For $t=0$, there is nothing to show as $\Gamma_{\cV(\emptyset)} = \id$. Let $t > 0$. We assume the claim holds for $t-1$. In the following, we implicitly use $\Gamma_{\cV(\bmx_t)} = \Gamma_{\cV(\bmx_{t-1})} \Gamma_{\cV(x_t)} = \Gamma_{\cV(x_t)} \Gamma_{\cV(\bmx_{t-1})}$; see \cite[Proposition~6.1]{Benson/Iyengar/Krause:2008}.

We apply $\ch{C}{\Gamma_{\cV(\bmx_t)} -}$ to \eqref{Kos_obj_Gamma:mor_u} with $s=t$ and take filtered colimits of the columns for $n\geq 1$. Using that filtered colimits of exact sequences of graded $R$-modules are exact (AB5), this yields a long exact sequence. Since $\ch{C}{\Gamma_{\cV(x_t)} N}$ is $(x_t)$-torsion for every object $N$ and $\Gamma_{\cV}$ is $R$-linear for any $\cV$, we obtain for the filtered of the third column
\begin{equation*}
\colim_n \ch{C}{\Gamma_{\cV(\bmx_t)} (\susp^{n|x_t|-t+1} \kosobj{M}{\bmx_{t-1}^n} )} \cong 0\,.
\end{equation*}
Hence we have an isomorphism of the filtered colimits between the first and second columns
\begin{equation*}
\colim_n \ch{C}{\Gamma_{\cV(\bmx_t)} (\susp^{-t} \kosobj{M}{\bmx_t^n})} \cong \colim_n \ch{C}{\Gamma_{\cV(\bmx_t)} (\susp^{-t+1} \kosobj{M}{\bmx_{t-1}^n})}\,.
\end{equation*}

It remains to show that the filtered colimit of the second column is isomorphic to $\ch{C}{\Gamma_{\cV(\bmx_t)} M}$. Applying any $R$-linear exact functor $\sF$ to \eqref{Kos_obj_Gamma:mor_u} yields again a diagram of the form \eqref{Kos_obj_Gamma:mor_u}, but with $M$ replaced by $\sF(M)$. Applying the induction hypothesis to the object $\Gamma_{\cV(x_t)}(M)$, we obtain for the filtered colimit of the second column
\begin{align*}
 \colim_n \ch{C}{\Gamma_{\cV(\bmx_{t-1})} (\susp^{-t+1} \kosobj{(\Gamma_{\cV(x_t)} M)}{\bmx_{t-1}^n})} 
&\cong \ch{C}{\Gamma_{\cV(\bmx_{t-1})} \Gamma_{\cV(x_t)}M} \\
&\cong \ch{C}{\Gamma_{\cV(\bmx_{t})}M}
\end{align*}
as desired.

\begin{step} \label{Gamma_hocolim:HH_Gamma_V} We have
$\colim_n \ch{C}{\Gamma_{\cV(\bmx_t)} V_{t,n}} \cong 0$ for any compact object $C$.
\end{step}

Applying $\colim_n \ch{C}{\Gamma_{\cV(\bmx_t)} -}$ to the morphism of exact triangles \cref{Kos_obj_Gamma:mor_v}, \cref{Gamma_hocolim:HH_Gamma_V} follows directly from \cref{Gamma_hocolim:HH_Gamma_Kos} as filtered colimits of exact sequences of graded $R$-modules are exact (AB5).

\begin{step}
There is an isomorphism $\hocolim_n V_{t,n} \to \rL_{\cV(\bmx_t)} M$ making the right square in \cref{Kos_obj_Gamma_can_mor} commute. 
\end{step}

Consider the following commutative diagram
\begin{equation*}
\begin{tikzcd}[column sep=small]
M \ar[r] \ar[d] & V_{t,\ell} \ar[r] \ar[d] & \hocolim_n V_{t,n} \ar[d] \\
\rL_{\cV(\bmx_t)} M \ar[r]& \rL_{\cV(\bmx_t)} V_{t,\ell} \ar[r] & \rL_{\cV(\bmx_t)} \hocolim V_{t,n} \nospacepunct{.} 
\end{tikzcd}
\end{equation*}
We will show that, with the exception of those morphisms originating from $M$ or $V_{t,\ell}$, all morphisms are isomorphisms. As $\cat{T}$ is compactly generated and $\Gamma_{\cV(\bmx_t)}$ commutes with homotopy colimits, \cref{Gamma_hocolim:HH_Gamma_V} implies $\Gamma_{\cV(\bmx_t)} \hocolim_n V_{t,n} = 0$. Hence the morphism
\begin{equation*}
\hocolim_n V_{t,n} \to \rL_{\cV(\bmx_t)} \hocolim_n V_{t,n}
\end{equation*}
is an isomorphism. 

Using \cite[Theorem~5.6]{Benson/Iyengar/Krause:2008} and \cite[Lemma~2.6]{Benson/Iyengar/Krause:2011b}, we determine the support
\begin{equation*}
\begin{aligned}
\supp_R(\rL_{\cV(\bmx_t)} & \kosobj{M}{\bmx_t^\ell}) = \supp_R(M) \cap \cV(\bmx_t^\ell) \cap (\Spec(R) \setminus \cV(\bmx_t)) = \emptyset
\end{aligned}
\end{equation*}
and deduce $\rL_{\cV(\bmx_t)} \kosobj{M}{\bmx_t^\ell} = 0$ by \cite[Corollary 5.7(1)]{Benson/Iyengar/Krause:2008}. Thus, $\rL_{\cV(\bmx_t)} M \to \rL_{\cV(\bmx_t)} V_{t,\ell}$ is an isomorphism. As isomorphisms are preserved under homotopy colimits and $\rL_{\cV(\bmx_t)}$ commutes with homotopy colimits by \cite[Corollary 6.5]{Benson/Iyengar/Krause:2008}, it follows that
\begin{equation*}
\rL_{\cV(\bmx_t)} M \to \rL_{\cV(\bmx_t)} \hocolim_n V_{t,n}
\end{equation*}
is an isomorphism. The composition of 
$\rL_{\cV(\bmx_t)} M \to \rL_{\cV(\bmx_t)} \hocolim_n V_{t,n}$ followed by the inverse of $\hocolim_n V_{t,n} \to \rL_{\cV(\bmx_t)} \hocolim_n V_{t,n}$ is an isomorphism making the right square in \cref{Kos_obj_Gamma_can_mor} commute as desired.

\begin{step}
There is an isomorphism $\hocolim_n \susp^{-t} \kosobj{M}{\bmx_t^n} \to \Gamma_{\cV(\bmx_t)} M$ making the left square in \cref{Kos_obj_Gamma_can_mor} commute.
\end{step}

By the compatibilities in \cref{Kos_obj_Gamma:mor_v}, there is an induced morphism
\begin{equation*}
\hocolim_n \susp^{-t} \kosobj{M}{\bmx_t^n} \to M\,,
\end{equation*}
and we obtain a commutative diagram 
\begin{equation} \label{Gamma_hocolim:iso_hocolim_Gamma}
\begin{tikzcd}
\susp^{-t} \kosobj{M}{(\bmx_t^\ell)} \ar[d] \ar[dr] \\
\hocolim_n \susp^{-t} \kosobj{M}{(\bmx_t^n)} \ar[r] & M \\
\Gamma_{\cV(\bmx_t)} \hocolim_n \susp^{-t} \kosobj{M}{(\bmx_t^n)} \ar[r] \ar[u] & \Gamma_{\cV(\bmx_t)} M \ar[u] \nospacepunct{.}
\end{tikzcd}
\end{equation}
The morphism
\begin{equation*}
\Gamma_{\cV(\bmx_t)} \hocolim_n \susp^{-t} \kosobj{M}{(\bmx_t^n)} \to \Gamma_{\cV(\bmx_t)} M
\end{equation*}
is an isomorphism by \cref{Gamma_hocolim:HH_Gamma_Kos} and the assumption that $\cat{T}$ is compactly generated.

As $\kosobj{M}{\bmx_t} \in \cat{T}_{\cV(\bmx_t)}$ by \cite[Proposition~2.11]{Benson/Iyengar/Krause:2011b} and $\cat{T}_{\cV(\bmx_t)}$ is closed under homotopy colimits, the morphism
\begin{equation*}
\Gamma_{\cV(\bmx_t)} \hocolim_n \susp^{-t} \kosobj{M}{\bmx_t^n} \to \hocolim_n \susp^{-t} \kosobj{M}{\bmx_t^n}
\end{equation*}
is an isomorphism as well by \cite[Corollary 5.7(1)]{Benson/Iyengar/Krause:2008}. Thus the zig-zag from $\hocolim_n \susp^{-t} \kosobj{M}{\bmx_t^n}$ to $\Gamma_{\cV(\bmx_t)} M$ is an isomorphism making the left square in \cref{Kos_obj_Gamma_can_mor} commute as desired. 
\end{proof}

\begin{remark}\label{compatibility_hocolim_gamma}
In the proof of \cref{Gamma_hocolim} we implicitly show that the induced morphism $\varphi\colon \hocolim_n \susp^{-t} \kosobj{M}{(x_1^n, \ldots, x_t^n)} \to M$, is unique up to isomorphism. That is, we show that for any $\phi$ there exists an isomorphism 
\begin{equation*}
\hocolim_n \susp^{-t} \kosobj{M}{(x_1^n, \ldots, x_t^n)} \xrightarrow{\cong} \Gamma_{\cV(x_1, \ldots, x_t)} M
\end{equation*}
making \cref{Gamma_hocolim:iso_hocolim_Gamma} together with the chosen $\varphi$ commute.
\end{remark}

\begin{remark}
For the homotopy colimit in \cref{Gamma_hocolim}, it is possible to replace the sequence $\{(n,\ldots,n)\}_{n \in \BN}$ by any other sequence in $\BN^t$ that is component-wise increasing and where each component diverges to infinity.
\end{remark}

%%%%%%%%%%%%%%%%%%%%%%%%%%%%%%%%%%%%%%%%%%%%%%%%%%
%%%%%%%%%%%%%%%%%%%%%%%%%%%%%%%%%%%%%%%%%%%%%%%%%%
\section{Regular elements}
Let $A$ be a commutative ring and $M$ an $A$-module. An element $x \in A$ is \emph{$M$-regular} if $x$ is a non-zero divisor on $M$. It is well-known that the following are equivalent
\begin{enumerate}
\item $x$ is $M$-regular; 
\item\label{firstKoszulHomology_zero} $\ch[1]{}{\Kos^A(x;M)} = 0$; 
\item\label{localCohomology_zero} $\lch[0]{(x)}{M} = 0$; 
\end{enumerate}
see for example \cite[Exercise~1.3.9]{Brodmann/Sharp:2013}. We can characterize conditions \eqref{firstKoszulHomology_zero} and \eqref{localCohomology_zero} in terms of the triangulated structure on the derived category of $A$-modules. These characterizations motivate our definition of a regular element in any $R$-linear triangulated category.

%%%%%%%%%%%%%%%%%%%%%%%%%%%%%%%%%%%%%%%%%%%%%%%%%%
\subsection{Ghost maps and vanishing of (co)homology} \label{ghost_map_vanishing_hh}

Let $\cat{T}$ be a triangulated category and $C \in \cat T$. A morphism $f \colon M \to N$ in $\cat{T}$ is called \emph{$C$-ghost} if the induced map
\begin{equation*}
\Hom[*]{\cat{T}}{C}{f} \colon \Hom[*]{\cat T}{C}{M}\to \Hom[*]{\cat T}{C}{N}
\end{equation*}
is zero. This means that pre-composition of $f$ with any morphism $\susp^n C \to M$ is zero for every integer $n$. Further, we say $f \colon M \to N$ is \emph{$t$-fold $C$-ghost}, if it is a $t$-fold composition of $C$-ghost morphisms. 

Let $C$ be a compact object in $\cat{T}$. Recall, one can think of $\ch[*]{C}{-} \colonequals \Hom[*]{\cat{T}}{C}{-}$ as a homology theory. Saying that $f \colon M \to N$ is $C$-ghost means that $f$ vanishes in this homology theory. The homology theory does not detect whether $f$ is a composition of multiple $C$-ghost morphisms. 

%%%%%%%%%%%%%%%%%%%%%%%%%%%%%%%%%%%%%%%%%%%%%%%%%%
\subsection{Over a commutative ring} \label{regular_elt_comm_ring}

Let $A$ be a commutative ring and let $\dcat{A}$ be the derived category of $A$-modules. Then, $\dcat{A}$ is an $A$-linear category, where we view $A$ as graded ring concentrated in degree $0$. Let $x \in A$ and $M$ an $A$-module, the latter is an object in $\dcat{A}$ when viewed as a complex concentrated in degree $0$. Then, $\kosobj{M}{x}$ coincides with the Koszul complex of $x$ with coefficients on $M$, denoted by $\Kos^A(x;M)$. For $C = A$, the functor $\ch{A}{-}$ is the classical homology functor of $A$-complexes. 

Ghost morphisms can be used to detect vanishing of Koszul homology in certain degrees. 

\begin{lemma} \label{Vanishing_Kos_hom}
Let $A$ be a commutative ring and $\bmx = x_1, \ldots, x_t$ a sequence of elements in $A$. For any $A$-module $M$ and $1 \leq s \leq t$, the following are equivalent
\begin{enumerate}
\item \label{Vanishing_Kos_hom:hom} $\ch[i]{}{\Kos^A(\bmx;M)} = 0$ for $i \geq s$; 
\item \label{Vanishing_Kos_hom:ghost} The truncation morphism $\Kos^A(\bmx; M) \to \Kos^A(\bmx; M)_{\geqslant s}$ is $A$-ghost in $\dcat{A}$; and 
\item The truncation morphism $\Kos^A(\bmx; M)_{\geqslant i-1} \to \Kos^A(\bmx; M)_{\geqslant i}$ is $A$-ghost in $\dcat{A}$ for all $s \leq i \leq t$. 
\end{enumerate}
If the above conditions hold, then the map $\Kos^A(\bmx; M) \to \susp^t M$ is $(t-s+1)$-fold $A$-ghost.
\end{lemma}
\begin{proof}
A truncation map $\tau \colon \Kos^A(\bmx;M)_{\geqslant i} \to \Kos^A(\bmx;M)_{\geqslant j}$ for $i \leq j$ induces maps on homology
\begin{equation*}
\ch[\ell]{}{\tau} \colon \ch[\ell]{}{\Kos^A(\bmx;M)_{\geqslant i}} \to \ch[\ell]{}{\Kos^A(\bmx;M)_{\geqslant j}}\,.
\end{equation*}
This map is an isomorphism when $\ell > j$, and it is injective when $\ell = j$. For the remaining indices, $\ell < j$, the map is zero.

Hence, the truncation map $\tau$ is $A$-ghost if and only if $\ch[\ell]{}{\tau} = 0$ for all $\ell \in \BZ$ if and only if $\ch[\ell]{}{\Kos^A(\bmx; M)_{\geqslant j}} = 0$ for $\ell \geq j$. This yields the claim.

The last assertion holds as $\Kos^A(\bmx; M)_{\geqslant t} = \susp^t M$. 
\end{proof}

If $A$ is noetherian and $\fa \subseteq A$ is an ideal, then we have an associated local cohomology functor whose vanishing in degree zero is related to the existence of ghost morphisms as well. The \emph{$\fa$-torsion functor} is defined as
\begin{equation*}
\sF_\fa(M) \colonequals \Set{m \in M | \text{there exists } n \geq 1 \text{ such that } \fa^n m = 0}\,.
\end{equation*}
Usually, this functor is denoted by $\Gamma_\fa$, we choose the notation $\sF_\fa$ to distinguish it from the local cohomology functor for triangulated categories. Having said that, by \cite[Section~9]{Benson/Iyengar/Krause:2008}, the right derived functor of $\sF_\fa$ coincides with the local cohomology functor; that is $R \sF_\fa \cong \Gamma_{\cV(\fa)}$. 

\begin{lemma} \label{Vanishing_local_coh}
Let $A$ be a commutative noetherian ring and $\fa \subseteq A$ an ideal. For any $A$-module $M$ the following are equivalent
\begin{enumerate}
\item \label{Vanishing_local_coh:coh} $\ch[0]{}{R \sF_{\fa} M} = 0$; and
\item \label{Vanishing_local_coh:ghost} $R \sF_{\fa} M \to M$ is $A$-ghost in $\dcat{A}$. 
\end{enumerate}
\end{lemma}
\begin{proof}
The implication \cref{Vanishing_local_coh:coh} $\implies$ \cref{Vanishing_local_coh:ghost} holds since the homology of $M$ is concentrated in degree zero, whereas the zeroth homology of $R \sF_\fa M$ is zero by assumption. 

We assume \cref{Vanishing_local_coh:ghost} holds. Let $I$ be an injective resolution of $M$; that is $M \xrightarrow{\sim} I$. Let $f \in \ch[0]{}{R \sF_\fa M}$; this element corresponds to a chain map $f \colon A \to \sF_\fa I$. As $R \sF_\fa M \to M$ is $A$-ghost, so is $\sF_\fa I \to I$. Hence, the composition $
A \to \sF_\fa I \to I$ is zero in $\dcat{A}$. For degree reasons, the map $A \xrightarrow{f_0} \sF_\fa I^0 \to I^0$ is zero. As $\sF_\fa I^0 \to I^0$ is injective, the map $A \xrightarrow{f_0} \sF_\fa I^0$ is zero. Hence, $f = 0$ and thus $\ch[0]{}{R \sF_\fa M} = 0$. 
\end{proof}

In both, \cref{Vanishing_Kos_hom,Vanishing_local_coh}, the first condition is not invariant under suspension while the second condition is invariant under suspension. The conditions need not be equivalent when $M$ is a complex. 

%%%%%%%%%%%%%%%%%%%%%%%%%%%%%%%%%%%%%%%%%%%%%%%%%%
\subsection{In triangulated categories} \label{regular_triangulated}

We return to the setting of an $R$-linear triangulated category $\cat{T}$. We fix an object $C\in \cat T$, which provides a cohomology theory. 

\begin{proposition} \label{regular_nzdivisor}
Let $R$ be a graded-commutative ring and $\cat{T}$ be an $R$-linear triangulated category. For $C,M \in \cat{T}$ and $x \in R$, the following are equivalent
\begin{enumerate}[start=0]
\item\label{regular_nzdivisor:Hom} $x$ is a non-zero divisor on $\ch[*]{C}{M}$; 
\item\label{regular_nzdivisor:Kosxghost} The map $\susp^{-1} \kosobj{M}{x} \to M$ is $C$-ghost;
\item\label{regular_nzdivisor:Kosxnghost} The map $\susp^{-1} \kosobj{M}{x^n} \to M$ is $C$-ghost for some $n \geq 1$; and
\item\label{regular_nzdivisor:allKosxnghost} The map $\susp^{-1} \kosobj{M}{x^n} \to M$ is $C$-ghost for all $n \geq 1$.
\end{enumerate}
If the ring $R$ is noetherian, the triangulated category $\cat{T}$ is compactly generated, and $C$ is compact, then the above conditions are equivalent to
\begin{enumerate}[resume]
\item\label{regular_nzdivisor:Gammaghost} The map $\Gamma_{\cV(x)} M \to M$ is $C$-ghost.
\end{enumerate}
\end{proposition}
\begin{proof}
The exact triangle defining $\kosobj{M}{x}$ induces a long exact sequence of $R$-modules
\begin{equation*}
\cdots \to \ch{C}{\susp^{-1} \kosobj{M}{x}} \to \ch{C}{M} \xrightarrow{x} \ch{C}{\susp^{|x|} M} \to \ch{C}{\kosobj{M}{x}} \to \cdots\,.
\end{equation*}
Then \cref{regular_nzdivisor:Hom} holds if and only if the map in the middle is injective. This holds if and only if the first map is zero. The latter is precisely \cref{regular_nzdivisor:Kosxghost}. 

The equivalences \cref{regular_nzdivisor:Hom} $\iff$ \cref{regular_nzdivisor:Kosxnghost} $\iff$ \cref{regular_nzdivisor:allKosxnghost} hold as $x$ is a non-zero divisor on $\ch{C}{M}$ if and only if $x^n$ is a non-zero divisor on $\ch{C}{M}$ for all/some $n \geq 1$. 

We now assume that $R$ is noetherian, $\cat{T}$ is compactly generated, and $C$ is compact. We make use of the connection of the Koszul objects and local cohomology described in \cref{Kos_obj_Gamma}. Using \cite[Lemma~3.4.3]{Krause:2022}, there is an isomorphism 
\begin{equation*}
\colim_n \ch{C}{\susp^{-1} \kosobj{M}{x^n}} \xrightarrow{\cong} \ch{C}{\Gamma_{\cV(x)} M}
\end{equation*}
and applying $\ch{C}{-}$ to the left-hand side of \eqref{Kos_obj_Gamma_can_mor} yields a commutative diagram
\begin{equation*}
\begin{tikzcd}
\ch{C}{\susp^{-1} \kosobj{M}{x^\ell}} \ar[r] \ar[dr] & \colim_n \ch{C}{\susp^{-1} \kosobj{M}{x^n}} \ar[r,"\cong"] \ar[d] & \ch{C}{\Gamma_{\cV(x)} M} \nospacepunct{.} \ar[dl] \\
& \ch{C}{M}
\end{tikzcd}
\end{equation*}
Thus \eqref{regular_nzdivisor:allKosxnghost} is equivalent to \eqref{regular_nzdivisor:Gammaghost}.
\end{proof}

We are now ready to make a sensible definition for a regular element on an object in a triangulated category.

\begin{definition} \label{dfb_regular_element}
Let $R$ be a graded-commutative ring and $\cat{T}$ be an $R$-linear triangulated category. We say an element $x \in R$ is \emph{$(C,M)$-regular}, where $C$ and $M$ are objects in $\cat{T}$, if $x$ is a non-zero divisor on $\ch[*]{C}{M}$.
\end{definition}

The notion of a $(A,M)$-regular element in $\dcat{A}$ recovers the classical notion of a $M$-regular element for an $A$-module $M$ since $\ch{A}{M} \cong M$. 

\begin{remark}
If $2\in R$ is not a zero-divisor and $x\in R$ is an element of odd degree, then $x^2=0$. Thus for any objects $C$, $M$ with $\ch[*]{C}{M}\neq 0$, elements of odd degree are not $(C,M)$-regular.
\end{remark}

\begin{remark} \label{regular_via_coghost}
There are characterizations of $(C,M)$-regular elements dual to \cref{regular_nzdivisor:Kosxghost,regular_nzdivisor:Kosxnghost,regular_nzdivisor:allKosxnghost} in \cref{regular_nzdivisor}. For $x \in R$ there is a commutative diagram
\begin{equation}\label{hom_x_in_first_or_second_variable}
\begin{tikzcd}[column sep=large]
\Hom[*]{\cat{T}}{C}{\susp^{-|x|}M}\ar[rr,"{\Hom[*]{\cat{T}}{C}{x(\susp^{-|x|} M)}}"] \ar[d,"\susp^{|x|}" swap] && \Hom[*]{\cat{T}}{C}{M}\\
\Hom[*]{\cat{T}}{\susp^{|x|}C}{M}\ar[rru,"{\Hom[*]{\cat{T}}{x(C)}{M}}"']&& 
\end{tikzcd}
\end{equation}
for any $C,M \in \cat{T}$. Hence $\susp^{-1} \kosobj{M}{x} \to M$ is $C$-ghost if and only if $\susp^{|x|} C \to \kosobj{C}{x}$ is $M$-coghost; see \cref{MM_regular} for more details. 

Alternatively, we observe that ghost and coghost maps are dual notions, meaning $f$ is $C$-ghost in $\cat{T}$ if and only if $\op{f}$ is $C$-coghost in $\op{\cat{T}}$. 
\end{remark}

%%%%%%%%%%%%%%%%%%%%%%%%%%%%%%%%%%%%%%%%%%%%%%%%%%
\subsection{Induced isomorphisms} \label{induced_iso_reg_elt}

Whenever $x$ is $(C,M)$-regular, the long exact sequence induced by the defining exact triangle of $\kosobj{M}{x}$ induces an isomorphism of graded $R$-modules
\begin{equation}\label{regular_Hom_Kos_obj_one_element}
\ch{C}{ M}/x \ch{C}{M} \xrightarrow{\cong} \ch{C}{\susp^{-|x|}\kosobj{M}{x}}
\end{equation}
sending $[C \to \susp^k M]$ to the composite $C\to \susp^k M\to \susp^{-|x|+k} \kosobj{M}{x}$.

Similarly, by \cref{regular_via_coghost}, we also obtain an isomorphism of graded $R$-modules
\begin{equation}\label{regular_Hom_Kos_obj_one_element_other_variable}
\Hom[*]{\cat{T}}{\susp^{-1}{\kosobj{C}{x}}}{M} \xleftarrow{\cong}\ch{C}{ M}/ \ch{C}{M}x 
\end{equation}
sending $[C\to \susp^k M]$ to the composite $\susp^{-1}\kosobj{C}{x}\to C \to \susp^k M$.

%%%%%%%%%%%%%%%%%%%%%%%%%%%%%%%%%%%%%%%%%%%%%%%%%%
%%%%%%%%%%%%%%%%%%%%%%%%%%%%%%%%%%%%%%%%%%%%%%%%%%
\section{Regular sequences}\label{sec:regular_sequences}

Over a commutative ring $A$, a sequence $x_1, \ldots, x_t$ is \emph{$M$-regular} for an $A$-module $M$, if
\begin{enumerate}
\item \label{class_reg_seq:ind} $x_i$ is $M/(x_1, \ldots, x_{i-1}) M$-regular for $1 \leq i \leq t$; and
\item \label{class_reg_seq:nonzero} $M/(x_1, \ldots, x_t) M \neq 0$. 
\end{enumerate}
A sequence that satisfies \cref{class_reg_seq:ind} but not necessarily \cref{class_reg_seq:nonzero} is called \emph{weakly $M$-regular}. 
We consider generalizations of these two conditions to triangulated categories separately.

%%%%%%%%%%%%%%%%%%%%%%%%%%%%%%%%%%%%%%%%%%%%%%%%%%
\subsection{Inductive condition}

We first generalize the inductive condition \cref{class_reg_seq:ind}. The characterization of regular elements from \cref{regular_nzdivisor} for triangulated categories suggests several possible choices of objects to use for the inductive definition, namely the domains of the $C$-ghost maps. We show that the definition is independent thereof.

\begin{proposition} \label{reg_seq_equiv}
Let $R$ be a graded-commutative ring and $\cat{T}$ an $R$-linear triangulated category. Let $C,M \in \cat{T}$ and $x \in R$ a $(C,M)$-regular element. For $y \in R$ the following are equivalent
\begin{enumerate}
\item \label{reg_seq_equiv:Kosx} $y$ is $(C,\kosobj{M}{x})$-regular;
\item \label{reg_seq_equiv:Kosxn} $y$ is $(C,\kosobj{M}{x^n})$-regular for some $n \geq 1$;
\item \label{reg_seq_equiv:allKosxn} $y$ is $(C,\kosobj{M}{x^n})$-regular for all $n \geq 1$.
\end{enumerate}
If the ring $R$ is noetherian, the triangulated category $\cat{T}$ is compactly generated and $C$ is compact, then the above conditions are equivalent to
\begin{enumerate}[resume]
\item \label{reg_seq_equiv:Gamma} $y$ is $(C,\Gamma_{\cV(x)} M)$-regular.
\end{enumerate}
\end{proposition}
\begin{proof}
Since $x$ is $(C,M)$-regular, $y$ is a non-zero divisor on $\ch{C}{\susp^{|x|} M}/x\ch{C}{M}$ if and only if $y$ is a non-zero divisor on $\ch{C}{\susp^{n|x|} M}/x^n \ch{C}{M}$ for some $n \geq 1$. Thus, we obtain the equivalences between \cref{reg_seq_equiv:Kosx,reg_seq_equiv:Kosxn,reg_seq_equiv:allKosxn} from \cref{regular_nzdivisor,regular_Hom_Kos_obj_one_element}.

We now assume that $R$ is noetherian, $\cat{T}$ is compactly generated and $C$ is compact. We use the characterization of regularity from \cref{regular_nzdivisor}\eqref{regular_nzdivisor:Gammaghost} to establish the equivalence of \cref{reg_seq_equiv:allKosxn} and \cref{reg_seq_equiv:Gamma}, that is, we show that the following are equivalent
\begin{enumprime}[start=3]
\item \label{reg_seq_equiv:allKosxn'} $\Gamma_{\cV(y)} \kosobj{M}{x^n} \to \kosobj{M}{x^n}$ is $C$-ghost for all $n \geq 1$;
\item \label{reg_seq_equiv:Gamma'} $\Gamma_{\cV(y)}\Gamma_{\cV(x)} M \to \Gamma_{\cV(x)} M$ is $C$-ghost.
\end{enumprime}
We will use that $\Gamma_{\cV(x)} \Gamma_{\cV(y)} = \Gamma_{\cV(x,y)} = \Gamma_{\cV(y)} \Gamma_{\cV(x)}$; see \cite[Proposition~6.1]{Benson/Iyengar/Krause:2008}. Since $\Gamma_{\cV(y)}$ is $R$-linear and commutes with homotopy colimits, we obtain a commutative diagram
\begin{equation*}
\begin{tikzcd}
\ch{C}{\Gamma_{\cV(x,y)} M} \ar[d] &[-3em] \cong &[-3em] \colim_n \ch{C}{\susp^{-1} \kosobj{\Gamma_{\cV(y)} M}{x^n}} \ar[d] & \ch{C}{\susp^{-1} \Gamma_{\cV(y)} \kosobj{M}{x^\ell}} \ar[d] \ar[l] \\
\ch{C}{\Gamma_{\cV(x)} M} & \cong & \colim_n \ch{C}{\susp^{-1} \kosobj{M}{x^n}} & \ch{C}{\susp^{-1} \kosobj{M}{x^\ell}} \ar[l]
\end{tikzcd}
\end{equation*}
using \cite[Proposition~2.9]{Benson/Iyengar/Krause:2011b}; see also \cref{Gamma_hocolim}.

The implication \cref{reg_seq_equiv:allKosxn'} $\implies$ \cref{reg_seq_equiv:Gamma'} follows, as the middle vertical map is zero by the universal property of the colimit, and hence the left vertical map is zero.

As the diagram
\begin{equation*}
\begin{tikzcd}
\ch{C}{\kosobj{M}{x^n}} \ar[r] \ar[d,"\cong"] & \ch{C}{\kosobj{M}{x^{n+1}}} \ar[d,"\cong"] \\
\ch{C}{M}/x^n \ch{C}{M} \ar[r,"x"] & \ch{C}{M}/x^{n+1} \ch{C}{M}
\end{tikzcd}
\end{equation*}
commutes and $x$ is a non-zero divisor on $\ch{C}{M}$, the horizontal maps are injective. Hence the map $\ch{C}{\susp^{-1} \kosobj{M}{x^\ell}} \to \colim_n \ch{C}{\susp^{-1} \kosobj{M}{x^n}}$ is injective for any $\ell$, and \cref{reg_seq_equiv:Gamma'} $\implies$ \cref{reg_seq_equiv:allKosxn'} follows.
\end{proof}

%%%%%%%%%%%%%%%%%%%%%%%%%%%%%%%%%%%%%%%%%%%%%%%%%%
\subsection{Non-zero condition} \label{subsec:nonzero}

We generalize the condition $M \neq (x_1, \ldots, x_t) M$ to the setting of a triangulated category. 

In the derived category $\dcat{A}$ of modules over a ring $A$, the canonical morphism $\kosobj{M}{x}=\Kos^A(x;M)\to M/xM$ is an isomorphism in $\dcat{A}$ when $x$ is an $(A,M)$-regular element. Hence, we are interested in the (non-)vanishing of the Koszul object.

For one element, there are numerous equivalent conditions.

\begin{proposition} \label{zero_one}
Let $R$ be a graded-commutative ring and $\cat{T}$ be an $R$-linear triangulated category. For $M \in \cat{T}$ and $x \in R$, the following are equivalent
\begin{enumerate}
\item \label{zero_one:Kos_map} The morphism $\susp^{-1} \kosobj{M}{x} \to M$ is zero;
\primeitem \label[enumprimei]{zero_one:Kos_obj} The object $\kosobj{M}{x}$ is zero;
\item \label{zero_one:Kosn_map} The morphism $\susp^{-1} \kosobj{M}{x^n} \to M$ is zero for some $n \geq 1$; 
\primeitem \label[enumprimei]{zero_one:Kosn_obj} The object $\kosobj{M}{x^n}$ is zero for some $n \geq 1$;
\item \label{zero_one:all_Kos_map} The morphism $\susp^{-1} \kosobj{M}{x^n} \to M$ is zero for all $n \geq 1$; and
\primeitem \label[enumprimei]{zero_one:all_Kos_obj} The object $\kosobj{M}{x^n}$ is zero for all $n \geq 1$.
\end{enumerate}
If the ring $R$ is noetherian and $\cat{T}$ is compactly generated, then the above conditions are equivalent to
\begin{enumerate}[resume]
\item \label{zero_one:Gamma_map} The morphism $\Gamma_{\cV(x)} M \to M$ is zero;
\primeitem \label[enumprimei]{zero_one:Gamma_obj} The object $\Gamma_{\cV(x)} M$ is zero.
\end{enumerate}
\end{proposition}
\begin{proof}
We first show the equivalence \cref{zero_one:Kos_map} $\iff$ \cref{zero_one:Kos_obj}. The backward direction is clear. If \cref{zero_one:Kos_map} holds, then there exists a morphism $f \colon \susp^{|x|} M \to M$ such that $f\circ x(M)=\id_M$. As $x$ is natural, we have
\begin{equation*}
x(M) \circ f = (\susp^{|x|} f) \circ x(\susp^{|x|} M) = (\susp^{|x|} f) \circ (-1)^{|x|^2} \susp^{|x|} x(M) = (-1)^{|x|} \id_{\susp^{|x|} M}\,.
\end{equation*}
Hence $x(M)$ is an isomorphism and $\kosobj{M}{x} = 0$.

The equivalences \cref{zero_one:Kosn_map} $\iff$ \cref{zero_one:Kosn_obj} and \cref{zero_one:all_Kos_map} $\iff$ \cref{zero_one:all_Kos_obj} follow by replacing $x$ with $x^n$. 

The implication \cref{zero_one:Kos_map} $\implies$ \cref{zero_one:Kosn_map} is clear. The converse holds as $\susp^{-1} \kosobj{M}{x} \to M$ factors through $\susp^{-1} \kosobj{M}{x^n} \to M$ by \cref{Kos_obj_Gamma:mor_v}.

The implication \cref{zero_one:all_Kos_obj} $\implies$ \cref{zero_one:Kos_obj} is clear. The converse follows by induction
using the exact triangle 
\begin{equation*}
\kosobj{M}{x} \to \kosobj{M}{x^n} \to \susp^{|x|} \kosobj{M}{x^{n-1}} \to \susp \kosobj{M}{x}
\end{equation*}
from \cref{kos_product}.

We now assume that $R$ is noetherian and $\cat{T}$ is compactly generated. The implication \cref{zero_one:Gamma_map} $\implies$ \cref{zero_one:Kos_map} holds as $\susp^{-1} \kosobj{M}{x} \to M$ factors through $\Gamma_{\cV(x)} M \to M$. Further, \cref{zero_one:all_Kos_obj} $\implies$ \cref{zero_one:Gamma_obj} by \cref{Gamma_hocolim} and \cref{zero_one:Gamma_obj} $\implies$ \cref{zero_one:Gamma_map} is clear. This completes the proof.
\end{proof}

The obstruction to generalize this \namecref{zero_one} to a sequence of elements $x_1, \ldots, x_t$ is that the Koszul object need not be functorial. As this assumption is only needed for the different powers of the sequence, we split the equivalences in multiple results to minimize the assumptions of each. 

\begin{proposition} \label{nonzero_Kos_map_equiv}
Let $R$ be a graded-commutative ring and $\cat{T}$ be an $R$-linear triangulated category.
Let $x_1, \ldots, x_t \in R$ be productive elements such that $x_i^n$ is Koszul-exact for each $1 \leq i \leq t$ and $n \geq 1$. For any $M \in \cat{T}$ the following are equivalent
\begin{enumerate}
\item \label{nonzero_Kos_map_equiv:Kos_map} The morphism $\susp^{-t} \kosobj{M}{(x_1, \ldots, x_t)} \to M$ is zero;
\primeitem \label[enumprimei]{nonzero_Kos_map_equiv:Kos_obj} The object $\kosobj{M}{(x_1, \ldots, x_t)}$ is zero;
\item \label{nonzero_Kos_map_equiv:Kosn_map} The morphism $\susp^{-t} \kosobj{M}{(x_1^n, \ldots, x_t^n)} \to M$ is zero for some $n \geq 1$; 
\primeitem \label[enumprimei]{nonzero_Kos_map_equiv:Kosn_obj} The object $\kosobj{M}{(x_1^n, \ldots, x_t^n)}$ is zero for some $n \geq 1$;
\item \label{nonzero_Kos_map_equiv:all_Kos_map} The morphism $\susp^{-t} \kosobj{M}{(x_1^n, \ldots, x_t^n)} \to M$ is zero for all $n \geq 1$; 
\primeitem \label[enumprimei]{nonzero_Kos_map_equiv:all_Kos_obj} The object $\kosobj{M}{(x_1^n, \ldots, x_t^n)}$ is zero for all $n \geq 1$.
\end{enumerate}
\end{proposition}

We first prove an auxiliary lemma for the equivalences ($i$) $\iff$ ($i$'). 

\begin{lemma} \label{zero_inductive_step}
Let $R$ be a graded-commutative ring and $\cat{T}$ be an $R$-linear triangulated category. We assume $x,y \in R$ are Koszul-exact with $x$ productive. For any $M \in \cat{T}$ the following are equivalent
\begin{enumerate}
\item \label{zero_inductive_step:Kos_map} The morphism $\susp^{-2} \kosobj{M}{(x,y)} \to M$ is zero; and
\item \label{zero_inductive_step:Kos_last_map} The morphism $\susp^{-2} \kosobj{M}{(x,y)} \to \susp^{-1} \kosobj{M}{x}$ is zero. 
\end{enumerate}
\end{lemma}
\begin{proof}
The implication \cref{zero_inductive_step:Kos_last_map} $\implies$ \cref{zero_inductive_step:Kos_map} is clear. We assume \cref{zero_inductive_step:Kos_map} holds. As $x$ and $y$ are Koszul-exact, there is a commutative diagram
\begin{equation*}
\begin{tikzcd}
\susp^{-1} \kosobj{M}{(x,y)} \ar[r] \ar[d] \ar[dr,"0"] & \susp^{-1} \kosobj{M}{y} \ar[r,"x"] \ar[d] & \susp^{|x|-1} \kosobj{M}{y} \ar[r] \ar[dl,dashed] & \susp^{-1} \kosobj{M}{(x,y)} \nospacepunct{,} \\
\susp^{-1} \kosobj{M}{x} \ar[r] & M 
\end{tikzcd}
\end{equation*}
where the top row is an exact triangle; see \cref{Kos_functor_compatible}. As the composition along the diagonal of the square is zero by assumption, we obtain a morphism $\susp^{|x|-1} \kosobj{M}{y} \to M$; that is
\begin{equation*}
(\susp^{-1} \kosobj{M}{y} \to M) = (\susp^{-1} \kosobj{M}{y} \xrightarrow{x} \susp^{|x|-1} \kosobj{M}{y} \to M)\,.
\end{equation*}
Applying $\kosobj{(-)}{x}$ to this identification, the right-hand side is the zero morphism as $x$ is productive. The left-hand side is, up to isomorphism, the morphism $\susp^{-1} \kosobj{M}{(x,y)} \to \kosobj{M}{x}$ by \cref{Kos_functor_compatible}. 
\end{proof}

\begin{proof}[Proof of \cref{nonzero_Kos_map_equiv}]
The implications \cref{nonzero_Kos_map_equiv:Kos_map} $\iff$ \cref{nonzero_Kos_map_equiv:Kos_obj}, \cref{nonzero_Kos_map_equiv:Kosn_map} $\iff$ \cref{nonzero_Kos_map_equiv:Kosn_obj} and \cref{nonzero_Kos_map_equiv:all_Kos_map} $\iff$ \cref{nonzero_Kos_map_equiv:all_Kos_obj} follow from \cref{zero_inductive_step,zero_one}.

The implications \cref{nonzero_Kos_map_equiv:all_Kos_map} $\implies$ \cref{nonzero_Kos_map_equiv:Kos_map} $\implies$ \cref{nonzero_Kos_map_equiv:Kosn_map} hold trivially. Further, the implication \cref{nonzero_Kos_map_equiv:Kosn_map} $\implies$ \cref{nonzero_Kos_map_equiv:Kos_map} holds as the morphism in the latter factors through the morphism of the former; see \cref{kos_product_diagram}. 
It remains to show \cref{nonzero_Kos_map_equiv:Kos_obj} $\implies$ \cref{nonzero_Kos_map_equiv:all_Kos_obj}. We assume $\kosobj{M}{(x_1, \ldots, x_t)}$ is zero. As the powers of the elements $x_1, \ldots, x_t$ are Koszul-exact, we obtain from \cref{kos_product} exact triangles
\begin{equation*}
\susp^{|x_i|-1} \kosobj{M}{\bmx^\bmalpha} \to \kosobj{M}{\bmx^{\bmalpha-(\alpha_i-1)\bme_i}} \to \kosobj{M}{\bmx^{\bmalpha+\bme_i}} \to \susp^{|x_i|} \kosobj{M}{\bmx^\bmalpha}
\end{equation*}
where $\bmalpha = (\alpha_1, \ldots, \alpha_t) \in \BN^t$ and $\bmx^\bmalpha = (x_1^{\alpha_1}, \ldots, x_t^{\alpha_t})$ and $\bme_i = (0, \ldots, 1, \ldots, 0)$. Hence, by induction, $\kosobj{M}{(x_1^n, \ldots, x_t^n)}$ is zero and thus \cref{nonzero_Kos_map_equiv:all_Kos_obj} holds.
\end{proof}

Finally, we obtain the connection of the canonical morphism for the Koszul object and the local cohomology:

\begin{proposition} \label{nonzero_gamma_Kos}
Let $R$ be a graded-commutative noetherian ring and $\cat{T}$ be a compactly generated $R$-linear triangulated category. For any $M \in \cat{T}$ and $x_1, \ldots, x_t \in R$ the following are equivalent
\begin{enumerate}[start=3]
\item \label{nonzero_gamma_Kos:all_Kos_map} The map $\susp^{-t} \kosobj{M}{(x_1^n, \ldots, x_t^n)} \to M$ is zero for all $n \geq 1$.
\item \label{nonzero_gamma_Kos:Gamma_map} The map $\Gamma_{\cV(x_1, \ldots, x_t)} M \to M$ is zero;
\primeitem \label[enumprimei]{nonzero_gamma_Kos:Gamma_obj} The object $\Gamma_{\cV(x_1, \ldots, x_t)} M$ is zero;
\item \label{nonzero_gamma_Kos:supp} $\supp_R(M) \cap \cV(x_1, \ldots, x_t) \neq \emptyset$. 
\end{enumerate}
\end{proposition}
\begin{proof}
The implication \cref{nonzero_gamma_Kos:Gamma_obj} $\implies$ \cref{nonzero_gamma_Kos:Gamma_map} is clear. For the converse, as $\rL_\cV$ and $\Gamma_\cV$ are associated localizing and colocalizing functors for any specialization closed set $\cV$, we know 
\begin{equation*}
\Gamma_{\cV(x_1, \ldots, x_t)} \rL_{\cV(x_1, \ldots, x_t)} = 0 \quad \text{and} \quad (\Gamma_{\cV(x_1, \ldots, x_t)})^2 = \Gamma_{\cV(x_1, \ldots, x_t)}\,;
\end{equation*}
see \cite[Proposition~4.12.1]{Krause:2010}. Hence 
\begin{equation*}
\Gamma_{\cV(x_1, \ldots, x_t)}(\Gamma_{\cV(x_1, \ldots, x_t)} M \to M)
\end{equation*}
is an isomorphism. So if $\Gamma_{\cV(x_1, \ldots, x_t)} M \to M$ is zero this means $\Gamma_{\cV(x_1, \ldots, x_t)} M = 0$. 

The equivalence \cref{nonzero_gamma_Kos:all_Kos_map} $\iff$ \cref{nonzero_gamma_Kos:Gamma_map} holds by the compatibilities \cref{compatibility_hocolim_gamma}, and the equivalence of \cref{nonzero_gamma_Kos:Gamma_map} $\iff$ \cref{nonzero_gamma_Kos:supp} follows from \cite[Corollary~5.7(2)]{Benson/Iyengar/Krause:2008}. 
\end{proof}

%%%%%%%%%%%%%%%%%%%%%%%%%%%%%%%%%%%%%%%%%%%%%%%%%%
\subsection{Definition of regular sequence}

\begin{figure}
\begin{center}
\begin{tikzcd}[/tikz/execute at end picture={
 \node (large) [rectangle, draw, dotted, fit=(4) (4p)]{};
 }]
\cref{nonzero_Kos_map_equiv:Kosn_map} \ar[r,Leftrightarrow] \ar[d,shift left,Rightarrow,dashed] & \cref{nonzero_Kos_map_equiv:Kos_map} \ar[d,shift left,Rightarrow,dashed] & \cref{nonzero_Kos_map_equiv:all_Kos_map} \ar[l,Rightarrow] \ar[d,shift left,Rightarrow,dashed] \ar[r,Leftrightarrow] & |[alias=4]| \cref{nonzero_gamma_Kos:Gamma_map} \ar[d,Leftrightarrow] \\
\cref{nonzero_Kos_map_equiv:Kosn_obj} \ar[u,shift left,Rightarrow] & \cref{nonzero_Kos_map_equiv:Kos_obj} \ar[u,shift left,Rightarrow] \ar[r,Rightarrow,dashed] & \cref{nonzero_Kos_map_equiv:all_Kos_obj} \ar[u,shift left,Rightarrow] & |[alias=4p]| \cref{nonzero_gamma_Kos:Gamma_obj}
\end{tikzcd}
\end{center}
\caption{The diagram depicts the implications proven in \cref{nonzero_Kos_map_equiv,nonzero_gamma_Kos}. The solid implications hold without any further assumptions, the dashed implications hold when the elements $x_i^n$ are Koszul-exact and productive, and the statements in the dotted box exist when $\cat{T}$ is compactly generated and $R$ noetherian.} \label{fig:implications_nonzero}
\end{figure}
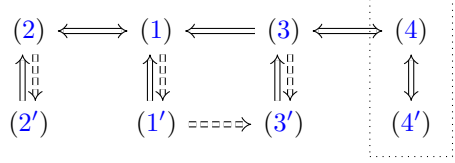

By analyzing the proofs of \cref{nonzero_Kos_map_equiv,nonzero_gamma_Kos}, we observe that the assumptions on the elements $x_1, \ldots, x_t$ in \cref{nonzero_Kos_map_equiv} are only used for some of the implications. We summarize the proven implications in \cref{fig:implications_nonzero}. In the \namecref{fig:implications_nonzero} it is marked which implications require which assumptions. We see, that the negation of \cref{nonzero_Kos_map_equiv:Kos_map} implies the negations of all the other statements. This explains the following definition:

\begin{definition} \label{defn_reg_sequence}
Let $R$ be a graded-commutative ring and $\cat{T}$ be an $R$-linear triangulated category. For $C,M \in \cat{T}$, we say a sequence $x_1, \ldots, x_t \in R$ is \emph{$(C,M)$-regular} if
\begin{enumerate}
\item \label{defn_reg_sequence:nzd} $x_s$ is $(C,\kosobj{M}{(x_1, \ldots, x_{s-1})})$-regular for $1 \leq s \leq t$; and
\item \label{defn_reg_sequence:nz} the morphism $\susp^{-t} \kosobj{M}{(x_1, \ldots, x_t)} \to M$ is non-zero.
\end{enumerate}
We say a sequence $x_1, \ldots, x_t$ is \emph{weakly $(C,M)$-regular} if it satisfies \cref{defn_reg_sequence:nzd}.
\end{definition}

Another reason we choose \cref{defn_reg_sequence:nz} is that this yields nice applications of regular sequences to level; see \cref{sec:application}.

While the definition of a regular element can be purely formulated for the $C$-homology $M$, the same does not holds for regular sequences. That is, non-vanishing need not be detected in the $C$-homology:

\begin{example}
Let $\cat{T} = \dcat{\BZ}$ and $R=\BZ$ and consider the complex
\begin{equation*}
M = (\dots \to 0 \to \BZ/3 \to \BZ/2 \to 0\to \dots)
\end{equation*}
with zero differentials. For the complex $C=\BZ/3$, concentrated in degree $0$, the element $x=2 \in \BZ$ is $(C,M)$-regular. We have $\ch{C}{\kosobj{M}{x}} = 0$, but $\kosobj{M}{x}\neq 0$.
Thus $x$ is a $(C,M)$-regular sequence, but $x$ is not a regular sequence for the graded module $\ch{C}{M}$.
\end{example}

Let $x_1, \ldots x_t$ be a $(C,M)$-regular sequence. By \cref{defn_reg_sequence}, each morphism in the non-zero composition
\begin{equation*}
\susp^{-t} \kosobj{M}{(x_1, \ldots, x_t)} \to \susp^{-t+1} \kosobj{M}{(x_1, \ldots, x_{t-1})} \to \cdots \to \susp^{-1} \kosobj{M}{x_1} \to M
\end{equation*}
is $C$-ghost. Moreover, when $R$ is noetherian, $\cat{T}$ compactly generated and $C$ compact, then by \cref{reg_seq_equiv,nonzero_gamma_Kos} the composition
\begin{equation*}
\Gamma_{\cV(x_1, \ldots, x_t)} M \to \Gamma_{\cV(x_1, \ldots, x_{t-1})} M \to \cdots \to \Gamma_{\cV(x_1)} M \to M
\end{equation*}
is non-zero and each morphism is $C$-ghost. In particular, both compositions are non-zero $t$-fold $C$-ghost map. However, it is not enough to assume that the compositions are non-zero $t$-fold $C$-ghost to obtain regularity. 

\begin{remark}
As in commutative algebra, one can define a notion of \emph{depth} as the maximal length of a $(C,M)$-regular sequence. This is a new notion and differs from other definitions of depth for complexes over a ring \cite{Foxby:1979,Iyengar:1999,Foxby/Iyengar:2003,Shaul:2020} or of depth in a triangulated category \cite{Liu/Liu/Yang:2016,Liu/Wang/Yang:2021}. We plan to explore this invariant in future work. 
\end{remark}

%%%%%%%%%%%%%%%%%%%%%%%%%%%%%%%%%%%%%%%%%%%%%%%%%%
%%%%%%%%%%%%%%%%%%%%%%%%%%%%%%%%%%%%%%%%%%%%%%%%%%
\section{Examples}

%%%%%%%%%%%%%%%%%%%%%%%%%%%%%%%%%%%%%%%%%%%%%%%%%%
\subsection{Derived category of modules over a commutative ring}

Let $A$ be a commutative ring. Then $A$, viewed as a graded ring concentrated in degree zero, acts on the derived category of $A$-modules $\dcat{A}$. As $\kosobj{M}{x} \cong \Kos^A(x;M) \cong \Kos^A(x;A) \otimes_A M$ for every $A$-complex $M$ and multiplication by $x$ is null-homotopic on $\Kos^A(x;A)$, every $x \in A$ is productive. In particular, the conditions in \cref{nonzero_Kos_map_equiv} are equivalent.

For an $A$-module $M$, a sequence $x_1, \ldots, x_t \in A$ is $(A,M)$-regular if and only if $x_1, \ldots, x_t$ is $M$-regular in the classical sense. This is clear from $\ch{}{M} \cong M$, and also from \cref{regular_elt_comm_ring}. For an $A$-complex $M$, a sequence $x_1, \ldots, x_t \in A$ is $(A,M)$-regular if and only if $x_1, \ldots, x_t$ is $\ch{}{M}$-regular and $\susp^{-t} M \to \Kos^A(x_1, \ldots, x_t;M)$ is non-zero in the derived category $\dcat{A}$. This notion differs from other notions of regular sequences on complexes. 

In \cite{Christensen:2001} an element is called regular on an $A$-complex $M$ if it is a non-zero divisor on the highest homology. There is also a slightly stronger definition, which is connected to the definition of depth from \cite{Foxby/Iyengar:2003}. While the initial condition for a regular element in \cite{Christensen:2001} differs from ours, the inductive definition coincides with ours, also using the Koszul complex. Christensen's notion of regular elements and regular sequences was generalized to dg modules by \cite{Minamoto:2021}.

While $A$ always acts on $\dcat{A}$, often there are other, bigger, (graded) rings acting on $\dcat{A}$, that provide longer regular sequences. 

%%%%%%%%%%%%%%%%%%%%%%%%%%%%%%%%%%%%%%%%%%%%%%%%%%
\subsection{Action of Hochschild cohomology} \label{HH_action}

Let $A$ be a commutative ring and $B$ a dg $A$-algebra. We write $\dcat{B}$ for the derived category of dg $B$-modules; that is the category with dg $B$-modules and morphisms of dg $B$-modules with the quasi-isomorphisms formally inverted. The Hochschild--Shukla cohomology of $B$ over $A$ is
\begin{equation*}
\HH{B/A} \colonequals \Ext[*]{B \lotimes_A \op{B}}{B}{B} = \Hom[*]{\dcat{B \lotimes_A \op{B}}}{B}{B}\,.
\end{equation*}
It is a graded-commutative ring that acts on $\dcat{B}$ via the tensor action
\begin{equation*}
- \lotimes_B - \colon \dcat{B \lotimes_A \op{B}} \times \dcat{B} \to \dcat{B}\,.
\end{equation*}
The Hochschild--Shukla cohomology and the tensor action are computed by choosing a free dg resolution of $B$ over $A$. When $A$ is a field, then Hochschild--Shukla cohomology is precisely Hochschild cohomology. 

\begin{example}
Let 
\begin{equation*}
E \colonequals \Kos^A(f_1, \ldots, f_c;A) = A \braket{\xi_1, \ldots, \xi_c | \partial(\xi_i) = f_i}
\end{equation*}
be the Koszul complex on a sequence $f_1, \ldots, f_c \in A$. The Koszul complex is free over $A$ and has a canonical dg $A$-algebra structure. By \cite[2.9]{Avramov/Buchweitz:2000}, a free dg algebra resolution of $E$ over $E \otimes_A \op{E}$ is given by
\begin{equation*}
(E \otimes_A \op{E})\braket{y_1, \ldots, y_c | \partial(y_i) = \xi_i \otimes 1 - 1 \otimes \xi_i} \xrightarrow{\sim} E\,.
\end{equation*}
Hence, one has
\begin{equation*}
\begin{aligned}
\HH{E/A} &= \Ext[*]{E \otimes_A \op{E}}{E}{E} = \ch{}{\Hom{E \otimes_A \op{E}}{(E \otimes_A \op{E})\braket{y_1, \ldots, y_c}}{E}} \\
&= \ch{}{E}[\chi_1, \ldots, \chi_c]
\end{aligned}
\end{equation*}
where $\chi_i$ is the dual of $y_i$. 

It is straightforward to check that $\kosobj{E}{\chi_i} \cong (E \otimes_A \op{E})\braket{y_1, \ldots, \hat{y}_i, \ldots, y_c}$. Hence $\chi_i(\kosobj{E}{\chi_i}) = 0$. In particular, \cref{nonzero_Kos_map_equiv} holds for $\chi_1, \ldots, \chi_c$, as well as for any of their products by \cref{productive_product}.

When $A$ is a regular local ring with residue field $k$, then, by \cite[Remark~3.2.6]{Pollitz:2019}, 
\begin{equation*}
\Hom[*]{\dcat{E}}{k}{k} = \Ext[*]{E}{k}{k} \cong k[\chi_1, \ldots, \chi_c] \otimes_k \bigwedge (\susp^{-1} k^{\dim A})
\end{equation*}
as left $\HH{E/A}$-modules. In particular, the sequence $\chi_1, \ldots, \chi_c$ is $(k,k)$-regular in $\dcat{E}$. 
\end{example}

\begin{example} \label{example_ci}
Let $A$ be a regular local ring and $f_1, \ldots, f_c$ an $A$-regular sequence and $B = A/\braket{f_1, \ldots, f_c}$. Then the Koszul complex $E = \Kos^A(f_1, \ldots, f_c;A)$  is a free dg algebra resolution of $B$ over $A$. Hence
\begin{equation*}
\HH{B/A} = \HH{E/A} = B[\chi_1, \ldots, \chi_c]
\end{equation*}
and $\chi_1, \ldots, \chi_c$ is $(k,k)$-regular in $\dcat{B}$. The elements $\chi_1, \ldots, \chi_c$ are precisely the cohomological operators; see \cite{Avramov/Sun:1998} for other equivalent descriptions.
\end{example}

Applying \cite[Theorem~B]{Letz:2025} to \cref{example_ci}, we obtain the following statement.

\begin{corollary}
Let $A$ be a regular local ring, $f_1, \ldots, f_c$ an $A$-regular sequence and $B = A/\braket{f_1, \ldots, f_c}$. The Rouquier dimension of $\dbcat{\mod{B}}$ is at least $c$.
\end{corollary}
\begin{proof}
For any $M, N$ in $\dbcat{\mod{B}}$, the bounded derived category of finitely generated $B$-modules, the graded module $\Hom[*]{\dcat{B}}{M}{N}$ is finitely generated over $B[\chi_1, \ldots, \chi_c]$; see \cite[5.1]{Avramov/Sun:1998}. The claim now follows from \cite[Theorem~B]{Letz:2025} as $\chi_1, \ldots, \chi_c$ is a $(k,k)$-regular sequence where $k$ is the residue field of $B$.
\end{proof}

%%%%%%%%%%%%%%%%%%%%%%%%%%%%%%%%%%%%%%%%%%\abcat%%%%%%%%
\subsection{Group algebras}

Let $G$ be a finite group and $k$ a field of characteristic $p>0$. The bounded derived category $\dbcat{\mod{kG}}$ is linear over the group cohomology ring $\ch{}{G,k}=\Ext[*]{kG}{k}{k}$ since it is tensor triangulated with unit $k$.
A $(k,k)$-regular sequence is an ordinary regular sequence in the graded-commutative ring $\ch{}{G,k}$; use \cref{prop:MM-regular} for the non-zero condition.

\begin{example}
Let $G=(\BZ/p)^r$ be an elementary abelian $p$-group of rank $r$. Its group cohomology over the field $k$ of characteristic $p$ is 
\begin{equation*}
 \ch{}{G,k} = k[t_1,\ldots,t_r]
\end{equation*}
with $|t_i|=1$ for $1\leq i\leq r$ if $p=2$, and
\begin{equation*}
 \ch{}{G,k} = k[x_1,\ldots,x_r]\otimes_k \Lambda(y_1,\ldots, y_r)
\end{equation*}
with $|x_i|=2$, $|y_i|=1$, for $1\leq i\leq r$ if $p$ is odd. In either case, the polynomial generators form a $(k,k)$-regular sequence of length $r$.
\end{example}

For $p=2$, the Rouquier dimension of $\dbcat{\mod{kG}}$ is known to be at least the $p$-rank of $G$; see \cite[Theorem ~4.9]{Rouquier:2006}. Using \cite[Theorem~B]{Letz:2025}, we extend this inequality to all primes $p$.

\begin{corollary}
Let $G$ be a finite group and and $k$ a field of characteristic $p>0$. The Rouquier dimension of $\dbcat{\mod{kG}}$ is at least the $p$-rank of $G$.
\end{corollary}
\begin{proof}
Let $E \subseteq G$ be an elementary abelian $p$-subgroup of maximal rank $r$. Thus $r$ is the $p$-rank of $G$. Since restriction and induction are exact, it follows from the Mackey formula that every object in $\dbcat{\mod{kE}}$ is a direct summand of the restriction of an object in $\dbcat{\mod{kG}}$. 
Hence, by \cite[Lemma~3.3]{Rouquier:2008}, the Rouquier dimension of $\dbcat{\mod{kG}}$ is at least the Rouquier dimension of $\dbcat{\mod{kE}}$.
The depth of $\Ext[*]{kE}{k}{k}$ is $r$. Hence \cite[Theorem~B]{Letz:2025} implies that the Rouquier dimension of $\dbcat{\mod{kG}}$ is at least the $p$-rank of $G$. 
\end{proof}
This improves the bound from \cite{Oppermann:2007,Bergh/Iyengar/Krause/Oppermann:2010} by one. 

\subsection{Classical productive elements}
Let $G$ be a finite group and $k$ a field of characteristic $p>0$. In general, there exist elements in $\ch{}{G,k}$ that are not $k$-productive as we will explain below after connecting the terminology to Carlson's notion of productive elements; see \cite{Carlson:1996}.

Let $\cat{T}=\stmod{kG}$ be the stable module category of finitely generated $kG$-modules. This is a tensor triangulated category with suspension $\Omega^{-1}$ and unit $k$. 
We consider $\stmod{kG}$ as an $R$-linear triangulated category where $R$ is the non-negative part of the graded endomorphism ring; that is $R=\ch{}{G,k}=\Ext[*]{kG}{k}{k}$ the group cohomology ring of $G$. 

We recall Carlson's construction of $L_{\zeta}$ for a non-zero element $\zeta\in \Ext[n]{kG}{k}{k}$ for $n>0$ and express it as a Koszul object. Computing $\Omega^n k$ via a minimal projective resolution, the class $\zeta$ is uniquely represented by a morphism of $kG$-modules $\Omega^n k \to k$ and $L_{\zeta}$ is defined to be the kernel of this morphism. We have $L_{\zeta}\cong \Omega^{n+1}\kosobj{k}{\zeta}$ in $\stmod{kG}$. Classically, the element $\zeta$ is called \emph{productive} if $\zeta\otimes_k L_{\zeta}=0$ in $\stmod{kG}$. Thus $\zeta$ is productive if and only if $\zeta$ is $k$-productive in our terminology. If $n$ is even and the characteristic of $k$ is odd, then $\zeta$ is productive by \cite[Theorem~4.1]{Carlson:1987}. 

The situation in characteristic $p=2$ differs. For $G$ the semidihedral group of order $16$ and $k$ a field of characteristic two, \cite[Remark following Proposition~5.9.6] {Benson:1998} states that there exists a class $\zeta\in \ch[1]{}{G,k}$ such that $\zeta^m$ is not productive for every $m>1$. In fact, the productive elements can be detected using power operations of Dyer--Lashof type; see \cite[Theorem~6.2]{Langer:2012}.

For a finite group $G$ and field $k$, there is an equivalence
\begin{equation*}
   \dbcat{\mod{kG}}/\dbcat{\proj{kG}} \xrightarrow{\sim} \stmod{kG}
\end{equation*}
of $\ch{}{G,k}$-linear triangulated categories; see \cite[Theorem~2.1]{Rickard:1989}. Taking a Verdier quotient preserves productive elements. Thus, non-productive elements in $\stmod{kG}$ are also non-productive in $\dbcat{\mod{kG}}$.

%%%%%%%%%%%%%%%%%%%%%%%%%%%%%%%%%%%%%%%%%%%%%%%%%%
\subsection{Finite-dimensional algebras}

Let $A$ be a self-injective finite-dimensional algebra over a field $k$ and $R$ a graded-commutative ring acting on $\cat{T}=\stmod{A}$ such that $\Hom[*]{\cat{T}}{C}{M}$ is noetherian over $R$ for any objects $C$, $M$ of $\cat{T}$. If $A$ has finite representation type, then the Rouquier dimension of $\stmod{A}$ is zero and there do not exist any $(M,M)$-regular sequences in $R$ by \cite[Theorem~B]{Letz:2025}. For example, the $k$-algebra $A = k[x]/x^n$ is self-injective, finite-dimensional and has finite representation type. Its center was fully described in \cite[Proposition~6.3, 6.4]{Krause/Ye:2011}.

%%%%%%%%%%%%%%%%%%%%%%%%%%%%%%%%%%%%%%%%%%\abcat%%%%%%%%
\subsection{Commutative ring spectra}

Let $S$ be the sphere spectrum and $R$ a commutative ring spectrum, that is, a commutative $S$-algebra in the terminology of \cite{Elmendorf/Kriz/Mandell/May:1997}. Let $\dcat{R}$ be the homotopy category of $R$-modules. This is a tensor triangulated category with unit $R$; see \cite[Example~1.2.3]{Hovey/Palmieri/Strickland:1997}. The graded endomorphism ring $\End[\dcat{R}]{*}{R}$ can be identified with the ring of stable homotopy groups $\pi_{*}(R)$ of $R$. In this context it is common to work with homological grading as in \cref{rem:homological_grading}, that is, with $\operatorname{Z}_{*}(\dcat{R})$ instead of $\ctr{\dcat{R}}$ so that $\dcat{R}$ is $\pi_{*}(R)$-linear instead of $\pi_{-*}(R)$-linear. 

The notion of regular sequence from \cite[Chapter~IV]{Elmendorf/Kriz/Mandell/May:1997} for an object $M\in \dcat{R}$ does not ask for a non-zero condition and thus corresponds to a weakly $(R,M)$-regular sequence in our terminology. For $M=R$, an $(R,R)$-regular sequence is an ordinary regular sequence in the graded-commutative ring $\pi_{*}(R)$; see \cref{prop:MM-regular}. We record two examples from the survey \cite{Richter:2022}.

\begin{example}
For the complex cobordism spectrum $MU$, one has $\pi_*(MU)\cong \BZ[x_1,x_2,\ldots]$ with $x_i$ in degree $2i$ and thus it admits arbitrary long $(MU,MU)$-regular sequences. 

For a prime $p$ and integer $n\geq 0$, let $E=E_{p,n}$ be the Morava $E$-theory spectrum. Its stable homotopy groups are
\begin{equation*}
    \pi_*(E)\cong W\BF_{p^n}\llbracket u_1,\ldots,u_{n-1}\rrbracket[u,u^{-1}],
\end{equation*}
where $W\BF_{p^n}$ denotes the ring of Witt vectors over $\BF_{p^n}$ and $|u_i|=0$, $|u|=2$. The sequence $(p,u_1,\ldots,u_{n-1})$ is $(E,E)$-regular.
\end{example}

%%%%%%%%%%%%%%%%%%%%%%%%%%%%%%%%%%%%%%%%%%%%%%%%%%
%%%%%%%%%%%%%%%%%%%%%%%%%%%%%%%%%%%%%%%%%%%%%%%%%%
\section{Properties}

In this section we collect various properties of regular sequences. 

%%%%%%%%%%%%%%%%%%%%%%%%%%%%%%%%%%%%%%%%%%%%%%%%%%
\subsection{Induced isomorphisms}

From the isomorphisms for a $(C,M)$-regular element in \cref{induced_iso_reg_elt}, we inductively obtain isomorphisms of graded $R$-modules
\begin{equation} \label{Hom_Koszul_object_in_different_variables}
\Hom[*]{\cat T}{\susp^{-t} \kosobj{C}{\bmx}}{M} \xleftarrow{\cong} \ch{C}{M}/\braket{\bmx}\ch{C}{M}\xrightarrow{\cong} \Hom[*]{\cat T}{C}{\susp^{-|\bmx|} \kosobj{M}{\bmx}} 
\end{equation} 
for any weakly $(C,M)$-regular sequence $\bmx = x_1, \ldots, x_t$ with $|\bmx|\coloneqq|x_1|+\ldots +|x_t|$. The first isomorphism sends a class represented by $C\to \susp^{k} M$ to the composite
\begin{equation*}
\susp^{-t} \kosobj{C}{\bmx} \to \ldots\to \susp^{-1}\kosobj{C}{x_1} \to C\to \susp^k M
\end{equation*}
and the second isomorphism sends the class to
\begin{equation*}
C\to \susp^k M \to \susp^{-|x_1|+k} \kosobj{M}{x_1}\to \ldots \to \susp^{-|\bmx|+k}\kosobj{M}{\bmx}\nospacepunct{.}
\end{equation*}

\begin{remark} \label{Hom_Kos_BIKP}
One might ask whether there exists an isomorphism 
\begin{equation*}
\Hom[*]{\cat T}{\susp^{-t} \kosobj{C}{\bmx}}{M} \cong \Hom[*]{\cat T}{C}{\susp^{-|\bmx|} \kosobj{M}{\bmx}}. 
\end{equation*}
of graded $R$-modules or just abelian groups without assuming that $\bmx$ is a $(C,M)$-regular sequence. A relation of this kind is stated in \cite[(3.1)]{Benson/Iyengar/Krause/Pevtsova:2021}, though no argument is provided. \cref{Counterexample:Hom_Kos_different_variable} in the appendix shows there need not be an isomorphism of graded $R$-modules in general.
\end{remark}

\begin{remark}
Let $C,M \in \cat{T}$ and $x_1, \ldots, x_t \in R$ a $(C,M)$-regular sequence. An element $y\in R$ is $(\kosobj{C}{(x_1, \ldots, x_t)},M)$-regular if and only if it is $(C, \kosobj{M}{(x_1, \ldots, x_t)})$-regular, by using the isomorphism \cref{Hom_Koszul_object_in_different_variables}. The non-zero condition cannot be transferred.
\end{remark}

Investigating an alternative non-zero condition, we observe that for a weakly $(C,M)$-regular sequence $\bmx = x_1, \ldots, x_t$ the following are equivalent:
\begin{enumerate}
\item \label{quotient_non_zero} The quotient $\ch{C}{M}/\braket{\bmx}\ch{C}{M}$ is zero;
\item \label{composite_coghost} The composite $\susp^{-t} \kosobj{C}{\bmx} \to \ldots\to \susp^{-1}\kosobj{C}{x_1} \to C$ is $M$-coghost; and
\item \label{composite_ghost} The composite $ M \to \susp^{-|x_1|} \kosobj{M}{x_1}\to \ldots \to \susp^{-|\bmx|}\kosobj{M}{\bmx}$ is $C$-ghost.
\end{enumerate}
The equivalence between \cref{quotient_non_zero} and \cref{composite_coghost} holds since a class represented by $C\to \susp^k M$ will be zero if and only if the composite 
\begin{equation*}
\susp^{-t} \kosobj{C}{\bmx} \to \ldots\to \susp^{-1}\kosobj{C}{x_1} \to C\to \susp^k M
\end{equation*}
is zero from the isomorphism \cref{Hom_Koszul_object_in_different_variables}. A similar argument shows the equivalence between \cref{quotient_non_zero} and \cref{composite_ghost}.

%%%%%%%%%%%%%%%%%%%%%%%%%%%%%%%%%%%%%%%%%%%%%%%%%%
\subsection{Classical local cohomology}

Let $R$ be a graded-commutative ring. As in \cref{regular_elt_comm_ring}, we let
\begin{equation*}
\sF_\fa(M) \colonequals \Set{m \in M | \text{there exists } n \geq 1 \text{ such that } \fa^n m = 0}\,.
\end{equation*}
be the (graded) $\fa$-torsion functor for an ideal $\fa \subseteq R$ and an $R$-module $M$. The local cohomology of $M$ is
\begin{equation*}
\lch{\fa}{M} \colonequals \ch{}{R \sF_\fa(I)}
\end{equation*}
when $M \to I$ is an injective resolution of $M$. For a regular sequence, there is a connection between the homology of the local cohomology functor and the classical local cohomology.

\begin{proposition}
Let $R$ be a noetherian graded-commutative ring and $\cat{T}$ a compactly generated triangulated category. For $C,M \in \cat{T}$ with $C$ compact and a $(C,M)$-regular sequence $x_1, \ldots, x_t$ one has
\begin{equation*}
\ch{C}{\Gamma_{\cV(x_1, \ldots, x_t)} M} \cong \lch[t]{\braket{x_1, \ldots, x_t}}{\ch{C}{\susp^{-t} M}}\,.
\end{equation*}
\end{proposition}
\begin{proof}
As $x_1, \ldots, x_t$ is $(C,M)$-regular sequence and $R$ is noetherian, using \cref{regular_nzdivisor,reg_seq_equiv}, the sequence $x_1^n, \ldots, x_t^n$ is weakly $(C,M)$-regular for all $n\in \mathbb{N}$. 
We have the following sequence of isomorphisms
\begin{equation*}
\begin{aligned}
\ch{C}{\Gamma_{\cV(x_1,\ldots,x_t)} M} &\cong \ch{C}{\hocolim_n \susp^{-t} \kosobj{M}{(x_1^n, \ldots, x_t^n)} }
\\
& \cong \colim_n \ch{C}{\susp^{-t} \kosobj{M}{(x_1^n, \ldots, x_t^n)}} \\
&\cong \colim_n \ch{C}{\susp^{d-t} M}/\braket{x_1^n, \ldots, x_t^n} \ch{C}{\susp^{d-t} M}
\end{aligned}
\end{equation*}
where $d \colonequals n(|x_1|+\cdots+|x_t|)$. 
The first isomorphism comes from \cref{Gamma_hocolim}, the second from \cref{hocolim_colim_compact}, and the last one from \cref{Hom_Koszul_object_in_different_variables}. 
We further obtain isomorphisms
\begin{equation*}
\begin{aligned}
\ch{C}{\susp^{d-t} M}/&\braket{x_1^n, \ldots, x_t^n} \ch{C}{\susp^{d-t} M} \\
&\cong \hh[0]{}{\Kos^R(x_1^n, \ldots, x_t^n;\ch{C}{\susp^{d-t} M})} \\
&\cong \hh[-t]{}{\Hom{R}{\Kos^R(x_1^n, \ldots, x_t^n; \ch{C}{\susp^{-t} M})}{R}}
\end{aligned}
\end{equation*} 
where the second isomorphism holds by the self-duality of the Koszul complex; see for example \cite[Proposition~1.6.10(d)]{Bruns/Herzog:1998}. 
Therefore, as the colimit commutes with homology, we get
\begin{equation*}
\colim_n \hh[-t]{}{\Hom{R}{\Kos^R(x_1^n, \ldots, x_t^n; \ch{C}{\susp^{-t} M})}{R}} \cong \lch[t]{\braket{x_1, \ldots, x_t}}{\ch{C}{\susp^{-t} M}}
\end{equation*}
using \cite[Theorem 3.5.6]{Bruns/Herzog:1998}. 
This finishes the proof. 
\end{proof}

%%%%%%%%%%%%%%%%%%%%%%%%%%%%%%%%%%%%%%%%%%%%%%%%%%
\subsection{Compatibility with triangulated structure}

By definition, an element $x$ is $(C,M)$-regular if and only if $x$ is $(\susp^i C, \susp^j M)$-regular for any $i,j \in \BZ$. The same holds for regular sequences. 

\begin{lemma}
Let $M' \to M \to M'' \to \susp M'$ be an exact triangle and $C \in \cat{T}$. Assume that $x \in R$ is $(C,M')$-regular and $(C,M'')$-regular. If the morphism $M'' \to \susp M'$ is $C$-ghost, then $x$ is also $(C,M)$-regular. 
\end{lemma}
\begin{proof}
As $M'' \to \susp M'$ is $C$-ghost, we obtain a short exact sequence of $R$-modules
\begin{equation*}
0 \to \ch{C}{M'} \to \ch{C}{M} \to \ch{C}{M''} \to 0\,.
\end{equation*}
Then the claim holds by \cite[Exercise~1.1.9]{Bruns/Herzog:1998}. 
\end{proof}

Without the assumption that $M'' \to \susp M'$ is $C$-ghost the statement need not hold.

\begin{example}
We view $\dcat{\BZ}$ as a $\BZ$-linear triangulated category. Then
\begin{equation*}
\BZ \to \BZ/2\BZ \to \susp \BZ \xrightarrow{-2} \susp \BZ
\end{equation*}
is an exact triangle where $2$ is $(\BZ,\BZ)$-regular but not $(\BZ,\BZ/2\BZ)$-regular. Note that $\susp \BZ \xrightarrow{-2} \susp \BZ$ is not $\BZ$-ghost.
\end{example}

%%%%%%%%%%%%%%%%%%%%%%%%%%%%%%%%%%%%%%%%%%%%%%%%%%
\subsection{Products}

We record the following property, which is a direct consequence of the classical setting.

\begin{lemma} \label{prod_weakly_regular}
Let $R$ be a graded-commutative ring and $\cat{T}$ an $R$-linear triangulated category and $C,M \in \cat{T}$. If $x_1, \ldots, x_t$ and $x_1, \ldots x'_i, \ldots, x_t$ are weakly $(C,M)$-regular sequences, then $x_1, \ldots x_i x'_i, \ldots x_t$ is also weakly $(C,M)$-regular. 
\end{lemma}
\begin{proof}
As the sequences $x_1, \ldots, x_t$ and $x_1, \ldots, x'_i, \ldots, x_t$ are weakly $(C,M)$-regular, they are weakly $\ch{C}{M}$-regular in the classical sense. Hence $x_1, \ldots, x_i x'_i, \ldots, x_t$ is weakly $\ch{C}{M}$-regular and thus weakly $(C,M)$-regular.
\end{proof}

For the same statement for regular sequences, we need extra assumptions.

\begin{lemma}
Let $R$ be a graded-commutative ring and $\cat{T}$ an $R$-linear triangulated category and $C,M \in \cat{T}$. Given $(C,M)$-regular sequences $x_1, \ldots,x_i,\ldots, x_t$ and $x_1, \ldots, x'_i, \ldots, x_t$ we assume
\begin{enumerate}
\item \label{prod_regular:via_kos_obj} $x_1, \ldots, x_t$ and $x'_i$ are Koszul-exact; or
\item \label{prod_regular:via_lch} $R$ is noetherian, $\cat{T}$ is compactly generated and $C$ compact. 
\end{enumerate}
Then the sequence $x_1, \ldots, x_i x'_i, \ldots x_t$ is also $(C,M)$-regular. 
\end{lemma}
\begin{proof}
It is immediate from \cref{prod_weakly_regular} that $x_1, \ldots, x_i x'_i,\ldots, x_t$ is weakly $(C,M)$-regular, so we only need to show the non-zero condition. We assume \cref{prod_regular:via_kos_obj} holds. Since the elements $x_1,\ldots, x_t$ and $x_i'$ are Koszul-exact, we may assume that $i=t$ by \cref{kos_obj_commute}. Using \cref{kos_product_diagram}, we obtain a commutative diagram
\begin{equation*}
\begin{tikzcd}
\susp^{-t} \kosobj{M}{(x_1, \ldots, x_t)} \ar[d] \ar[dr] \\
\susp^{-t} \kosobj{M}{(x_1, \ldots, x_{t-1}, x_tx'_t)} \ar[r] & \susp^{-t+1} \kosobj{M}{(x_1, \ldots, x_{t-1})} \ar[r] & M \nospacepunct{.}
\end{tikzcd}
\end{equation*}
Hence, if $\susp^{-t} \kosobj{M}{(x_1, \ldots, x_t)} \to M$ is non-zero, so is $\susp^{-t} \kosobj{M}{(x_1, \ldots, x_{t-1}, x_tx'_t)} \to M$. 

We assume \cref{prod_regular:via_lch} holds. From \cite[Proposition~6.1(1)]{Benson/Iyengar/Krause:2008}, for every $\mathcal{V}\subseteq\mathcal{U}$, we have that $\Gamma_{\cV} M= \Gamma_{\cV} \Gamma_{\mathcal{U}} M$. Hence, if $\Gamma_{\cV(x_1, \ldots, x_t)} M$ is non-zero, then so is $\Gamma_{\cV(x_1, \ldots, x_i x'_i, \ldots, x_t)} M$, since 
\begin{equation*}
\cV(x_1, \ldots, x_i x'_i, \ldots , x_t)=\cV(x_1, \ldots, x_t) \cup \cV(x_1, \ldots, x'_i, \ldots, x_t)\,. \qedhere
\end{equation*}
\end{proof}

%%%%%%%%%%%%%%%%%%%%%%%%%%%%%%%%%%%%%%%%%%%%%%%%%%
\subsection{Localization}

Let $\fp$ be a prime ideal in $R$. By \cite[Theorem~4.7]{Benson/Iyengar/Krause:2008}, the object $L_{\cZ(\fp)} M$ can be viewed as the localization of $M$ at the point $\fp$. 

\begin{lemma}\label{ghost_maps_under_localization}
Let $R$ be a noetherian graded-commutative ring, $\cat{T}$ a compactly generated $R$-linear triangulated category and $C \in \cat{T}$ a compact object. For a map $M\to N$ in $\cat{T}$, the following are equivalent
\begin{enumerate}
\item The map $M\to N$ is $C$-ghost;
\item The map $L_{\cZ(\fp)} M \to L_{\cZ(\fp)} N$ is $C$-ghost for all prime ideals $\fp$; and
\item The map $L_{\cZ(\fm)} M \to L_{\cZ(\fm)} N$ is $C$-ghost for all maximal ideals $\fm$.
\end{enumerate} 
\end{lemma}
\begin{proof}
For $\fp \in \Spec R$, we consider the following commutative diagram
\begin{equation*}
\begin{tikzcd}
\ch{C}{M}_\fp \ar[r, "\cong"] \ar[d] & \ch{C}{ L_{\mathcal{Z}(\fp)} M} \ar[d] \\
\ch{C}{N}_\fp \ar[r,"\cong"] & \ch{C}{L_{\mathcal{Z}(\fp)} N} \nospacepunct{.}
\end{tikzcd}
\end{equation*}
By \cite[Theorem~4.7]{Benson/Iyengar/Krause:2008}, the horizontal maps of the diagram are isomorphisms. Hence, it is enough to observe that the module homomorphism $\ch{C}{M}\to \ch{C}{N}$ is zero, if and only if $\ch{C}{M}_\fp \to \ch{C}{N}_\fp$ is zero for all prime ideals of $R$, if and only if $\ch{C}{M}_\fm \to \ch{C}{N}_\fm$ is zero for all maximal ideals of $R$, which completes the proof.
\end{proof}

\begin{lemma}\label{weakly_regular_under_localization}
Let $R$ be a noetherian graded-commutative ring and $\cat{T}$ a compactly generated $R$-linear triangulated category. For objects $C,M \in \cat{T}$ with $C$ compact and a sequence $x_1, \ldots, x_t \in R$, the following are equivalent
\begin{enumerate}
\item \label{regular_under_localization:non_localized} $x_1, \ldots, x_t$ is a weakly $(C,M)$-regular sequence;
\item \label{regular_under_localization:prime_ideals} $x_1, \ldots, x_t$ is a weakly $(C,L_{\mathcal{Z}(\fp)} M)$-regular sequence for all $\fp \in \supp_R (M)$; and
\item \label{regular_under_localization:maximal_ideals} $x_1, \ldots, x_t$ is a weakly $(C,L_{\mathcal{Z}(\fm)} M)$-regular sequence for all maximal ideals $\fm \in \supp_R (M)$.
\end{enumerate}
\end{lemma}
\begin{proof}
For convenience, we write $\bmx_s = x_1, \ldots, x_s$ for any $1 \leq s \leq t$. By \cref{ghost_maps_under_localization} and using $\Gamma_{\cV(\bmx_{s})} L_{\mathcal{Z}(\fp)} \cong L_{\mathcal{Z}(\fp)} \Gamma_{\cV(\bmx_{s})}$ and $\Gamma_{\cV(\bmx_{s})}\cong\Gamma_{\cV(x_s)} \Gamma_{\cV(\bmx_{s-1})}$ of \cite[Proposition~6.1]{Benson/Iyengar/Krause:2008} for every $1\leq s\leq t$, the equivalences hold. 
\end{proof}

\begin{lemma}\label{regular_under_localization}
Let $R$ be a noetherian graded-commutative ring and $\cat{T}$ a compactly generated $R$-linear triangulated category. Let $C,M \in \cat{T}$ with $C$ compact. If a sequence $x_1, \ldots, x_t$ is $(C,M)$-regular, then it is also $(C,L_{\mathcal{Z}(\fp)} M)$-regular for all $\fp \in \supp_R (M)\cap \cV(x_1,\ldots, x_t)$.
\end{lemma}
\begin{proof}
By \cref{weakly_regular_under_localization}, it remains to show that $\Gamma_{\cV(x_1, \ldots, x_t)}L_{\mathcal{Z}(\fp)} M$ is non-zero for $\fp \in \supp_R (M)\cap \cV(x_1,\ldots, x_t)$. This is immediate from \cite[Theorem 5.2 and Theorem 5.6]{Benson/Iyengar/Krause:2008}.
\end{proof}

%%%%%%%%%%%%%%%%%%%%%%%%%%%%%%%%%%%%%%%%%%%%%%%%%%
\subsection{\texorpdfstring{$(M,M)$}{(M,M)}-regular sequences} \label{MM_regular}

In \cref{regular_via_coghost} we observed that there is a symmetry between $C$ and $M$ for $(C,M)$-regular elements. This extends to a symmetry for weakly $(C,M)$-regular sequences. We make this symmetry precise here. 

Let $\cat{T}$ be a triangulated category and $C \in \cat{T}$. A morphism $f \colon M \to N$ in $\cat{T}$ is called \emph{$C$-coghost} if the induced map
\begin{equation*}
\Hom[*]{\cat{T}}{f}{C} \colon \Hom[*]{\cat{T}}{N}{C} \to \Hom[*]{\cat{T}}{M}{C}
\end{equation*}
is zero.

\begin{lemma} \label{weakly_regular_coghost}
Let $R$ be a graded-commutative ring and $\cat{T}$ an $R$-linear triangulated category. For $C,M \in \cat{T}$ and $x_1, \ldots, x_t \in R$, the following are equivalent
\begin{enumerate}
\item $x_1, \ldots, x_t$ is a weakly $(C,M)$-regular sequence; and
\item $\susp^{|x_s|} \kosobj{C}{(x_1, \ldots, x_{s-1})} \to \kosobj{C}{(x_1, \ldots, x_s)}$ is $M$-coghost for any $1 \leq s \leq t$.
\end{enumerate}
\end{lemma}
\begin{proof}
By \cref{hom_x_in_first_or_second_variable} and the long exact sequence in homology for the Koszul object of $C$, an element $x\in R$ is $(C,M)$-regular if and only if
\[\Hom[*]{\cat{T}}{\kosobj{C}{x}}{M}\to \Hom[*]{\cat{T}}{\susp^{|x|}C}{M}\]
is zero, that is, if and only if $\susp^{|x_1|}C\to \kosobj{C}{x}$ is $M$-coghost. The statement now follows by induction.
%We use induction on $t$ to construct an isomorphism 
%\begin{equation*}
%\Hom{\cat{T}}{C}{\susp^{-t} \kosobj{M}{(x_1, \ldots, x_t)}} \xrightarrow{\cong} \Hom{\cat{T}}{\susp^{-|x_1|-\cdots-|x_t|} \kosobj{C}{(x_1, \ldots, x_t)}}{M}\,.
%\end{equation*}
%For $t = 0$ we take the identity map. For $t > 1$, we assume we have an isomorphism for $t-1$. We consider the commutative diagram
%\begin{equation*}
%\begin{tikzcd}
%\Hom{\cat{T}}{C}{\susp^{|x_t|-t} \kosobj{M}{\bmx_{t-1}}} \ar[d] \ar[r,"\cong"] & \Hom{\cat{T}}{\susp^{-|x_1|-\cdots-|x_t|+1} \kosobj{C}{\bmx_{t-1}}}{M} \ar[d] %\\
%\Hom{\cat{T}}{C}{\susp^{-t} \kosobj{M}{\bmx_t}} \ar[d] \ar[r,dashed,"\cong"] & \Hom{\cat{T}}{\susp^{-|x_1|-\cdots-|x_t|} \kosobj{C}{\bmx_t}}{M} \ar[d] \\
%\Hom{\cat{T}}{C}{\susp^{-t+1} \kosobj{M}{\bmx_{t-1}}} \ar[d,"x_t"] \ar[r,"\cong"] & \Hom{\cat{T}}{\susp^{-|x_1|-\cdots-|x_{t-1}|} \kosobj{C}{\bmx_{t-1}}}{M} \ar[d,"x_t"] \\
%\Hom{\cat{T}}{C}{\susp^{|x_t|-t+1} \kosobj{M}{\bmx_{t-1}}} \ar[r,"\cong"] & \Hom{\cat{T}}{\susp^{-|x_1|-\cdots-|x_t|} \kosobj{C}{\bmx_{t-1}}}{M}
%\end{tikzcd}
%\end{equation*}
%where $\bmx_t = x_1, \ldots, x_t$ and $\bmx_{t-1} = x_1, \ldots, x_{t-1}$ and each column is a long exact sequence. The solid horizontal morphisms exist and are isomorphisms by the induction hypothesis. Hence the dashed horizontal morphism exists and is an isomorphism. 

%By the middle square in the above diagram the morphism $\susp^{-t} \kosobj{M}{\bmx_t} \to \susp^{-t+1} \kosobj{M}{\bmx_{t-1}}$ is $C$-ghost if and only if $\kosobj{C}{\bmx_{t-1}} \to \susp^{|x_t|} \kosobj{C}{\bmx_t}$ is $M$-coghost. 
\end{proof}

This symmetry need not extend to the $(C,M)$-regular sequences in general. However, when $C=M$ we obtain a particular nice characterization.

\begin{proposition}\label{prop:MM-regular}
The following are equivalent
\begin{enumerate}
\item $x_1, \ldots, x_t$ is an $(M,M)$-regular sequence; 
\item $x_1, \ldots, x_t$ is a $\Hom[*]{\cat{T}}{M}{M}$-regular sequence in the classical sense;
\item Each morphism in the composition 
\begin{equation*}
\susp^{-t} \kosobj{M}{(x_1, \ldots, x_t)} \to \susp^{-t+1} \kosobj{M}{(x_1, \ldots, x_{t-1})} \to \cdots \to M
\end{equation*}
is $M$-ghost and the composition is non-zero; and
\item Each morphism in the composition
\begin{equation*}
\susp^{|x_1|+\cdots+|x_t|} M \to \susp^{|x_2|+\cdots+|x_t|} \kosobj{M}{x_1} \to \cdots \to \kosobj{M}{(x_1, \ldots, x_t)}
\end{equation*}
is $M$-coghost and the composition is non-zero.
\end{enumerate}
\end{proposition}
\begin{proof}
By \cref{weakly_regular_coghost,regular_nzdivisor}, it is enough to show the following are equivalent
\begin{enumprime}
\item\label{nonzero_1} $\susp^{-t} \kosobj{M}{(x_1, \ldots, x_t)} \to M$ is non-zero;
\item\label{nonzero_2} $\braket{x_1, \ldots, x_t} \Hom[*]{\cat{T}}{M}{M} \neq \Hom[*]{\cat{T}}{M}{M}$; 
\item\label{nonzero_3} $\susp^{-t} \kosobj{M}{(x_1, \ldots, x_t)} \to M$ is non-zero; 
\item\label{nonzero_4} $\susp^{|x_1|+\ldots+|x_t|} M \to \kosobj{M}{(x_1, \ldots, x_t)}$ is non-zero;
\end{enumprime}
when $x_1, \ldots, x_t$ is weakly $\Hom[*]{\cat{T}}{M}{M}$-regular in the classical sense.

The morphisms in \eqref{nonzero_1} and \eqref{nonzero_3} are the same. The morphism from \eqref{nonzero_1} corresponds to the class of $\id_M$ in $\ch{M}{M}/\braket{x_1,\ldots,x_t}\ch{M}{M}$ under the first isomorphism in \cref{Hom_Koszul_object_in_different_variables}. Under the second isomorphism in \cref{Hom_Koszul_object_in_different_variables}, this corresponds to a suspension of the morphism in \eqref{nonzero_4}. Hence all four non-zero conditions are equivalent.
\end{proof}

%%%%%%%%%%%%%%%%%%%%%%%%%%%%%%%%%%%%%%%%%%%%%%%%%%
%%%%%%%%%%%%%%%%%%%%%%%%%%%%%%%%%%%%%%%%%%%%%%%%%%
\section{Application: Level} \label{sec:application}

Finally, we use regular sequences to obtain lower bounds for level. 

%%%%%%%%%%%%%%%%%%%%%%%%%%%%%%%%%%%%%%%%%%%%%%%%%%
\subsection{Level}

Let $\cat{T}$ be a triangulated category and $M \in \cat{T}$. We denote the smallest thick subcategory of $\cat{T}$ containing $M$ by $\thick_\cat{T}(M)$; it is called the thick subcategory \emph{generated by $M$}. It is stratifed by the $M$-level.

The \emph{$M$-level} of $N \in \cat{T}$ is defined inductively: When $N \cong 0$, then $\level^M_{\cat{T}}(N)=0$. When $N$ can be obtained from $M$ using (de)suspensions, retracts and finite coproducts, then $\level^M_{\cat{T}}(N)=1$. Otherwise, it is
\begin{equation*}
\level^M_{\cat{T}}(N) \colonequals \inf \Set{n\geq 0\ | \begin{gathered}
\text{ there is an exact triangle} \\
N' \to N \oplus \tilde{N} \to N'' \to \susp N' \\
\text{with} \level^M_{\cat{T}}(N')=1 \text{ and } \level^M_{\cat{T}}(N'')= n-1
\end{gathered}}\,.
\end{equation*}
This invariant was introduced in \cite[2.3]{Avramov/Buchweitz/Iyengar/Miller:2010} and is based on the construction in \cite[Section~2.2]{Bondal/VanDenBergh:2003}. 

%%%%%%%%%%%%%%%%%%%%%%%%%%%%%%%%%%%%%%%%%%%%%%%%%%
\subsection{Ghost lemma} \label{ghost_lemma}

The main tool to obtain lower bounds of level is the ghost lemma. If there is a non-zero $t$-fold composition of $C$-ghost morphisms starting at $M$, then $t < \level_\cat{T}^C(M)$. There are various versions of this result; see for example \cite[Lemma~2.2]{Beligiannis:2008}. 

%%%%%%%%%%%%%%%%%%%%%%%%%%%%%%%%%%%%%%%%%%%%%%%%%%
\subsection{Length of a regular sequence as lower bound}

From the definition of a regular sequence via ghost morphisms and the ghost lemma, one immediately obtains a lower bound for level: Let $x_1, \ldots, x_t$ be a $(C,M)$-regular sequence. Then
\begin{equation*}
\level_\cat{T}^C(\kosobj{M}{(x_1, \ldots, x_t)}) \geq t+1\,.
\end{equation*}

We obtain a similar bound for all objects in a certain subcategory. The inequalities are connected, as every weakly $(C,M)$-regular is weakly $(C,C)$-regular.

\begin{theorem} \label{level_lower_bound_reg}
Let $R$ be a graded-commutative ring and $\cat{T}$ an $R$-linear triangulated category. Let $C,M \in \cat{T}$ and $x_1, \ldots, x_t \in R$ a $(C,C)$-regular sequence. If $M \in \thick_\cat{T}(\kosobj{C}{(x_1, \ldots, x_t)})$ and $M\neq 0$, then
\begin{equation*}
\level_\cat{T}^C(M) \geq t + 1\,.
\end{equation*}
\end{theorem}
\begin{proof}
For convenience we write $\bmx_s^n$ for the sequence $x_1^n, \ldots, x_s^n$ for any $1 \leq s \leq t$. 

As $M$ lies in the thick subcategory generated by $\kosobj{C}{(x_1, \ldots, x_t)}$, the elements $x_1, \ldots, x_t$ act nilpotently on $M$. That is there exists an integer $n \geq 1$ such that $x_s^n(M) = 0$ for all $1 \leq s \leq t$. Moreover, $M$ also lies in the thick subcategory generated by $C$. Hence there exists a non-zero morphism $\susp^d M \to C$ for some integer $d$. Using induction on $t$, we show that $\susp^d M \to C$ factors through $\susp^{-t} \kosobj{C}{\bmx_t^n} \to C$. For $t = 0$ there is nothing to show. For $t > 0$, we obtain the commutative diagram
\begin{equation*}
\begin{tikzcd}
& \susp^d M \ar[d] \ar[dr,"0"] \ar[dl,dashed] \\
\susp^{-t} \kosobj{C}{\bmx_t^n} \ar[r] & \susp^{-t+1} \kosobj{C}{(\bmx_{t-1}^n)} \ar[r,"x_t^n"] & \susp^{n|x_t|-t+1} \kosobj{C}{\bmx_{t-1}^n} \ar[r] & \susp^{-t+1} \kosobj{C}{\bmx_t^n} \nospacepunct{,}
\end{tikzcd}
\end{equation*}
where the bottom row is an exact triangle. As $x_t^n$ acts trivially on $M$, we obtain the desired morphism $\susp^d M \to \susp^{-t} \kosobj{C}{\bmx_t^n}$. In particular, we obtain a non-zero composition
\begin{equation*}
\susp^d M \to \susp^{-t} \kosobj{C}{(x_1^n, \ldots, x_t^n)} \to C\,.
\end{equation*}
As $x_1^n, \ldots, x_t^n$ is $(C,C)$-regular, the latter morphism is $t$-fold $C$-ghost. The claim now follows from the ghost lemma; see for example \cite[Lemma~2.2]{Beligiannis:2008}.
\end{proof}

This bound is optimal, as for $M = C$ this is an equality whenever one has an enhancement; cf.\@ \cite[Section~5]{Letz/Stephan:2025}.

\begin{corollary}
Let $R$ be a graded-commutative ring and $\cat{T}$ an $R$-linear topological triangulated category. We assume $R$ acts topologically functorial on $\cat{T}$.

If $x_1, \ldots, x_t \in R$ is a $(C,C)$-regular sequence, then
\begin{equation*}
\level^C(\kosobj{C}{(x_1, \ldots, x_t)}) = t+1\,.
\end{equation*}
\end{corollary}
\begin{proof}
The inequality $\leq$ holds by \cite[Theorem~B]{Letz/Stephan:2025} and $\geq$ by \cref{level_lower_bound_reg}. 
\end{proof}

%%%%%%%%%%%%%%%%%%%%%%%%%%%%%%%%%%%%%%%%%%%%%%%%%%
\subsection{Connection to support}

There are classification of thick subcategories in various triangulated categories. In the presence of such a classification the condition $M \in \thick_\cat{T}(\kosobj{C}{(x_1, \ldots, x_t)})$ in \cref{level_lower_bound_reg} can be rephrased. 

In various settings there are such classifications in terms of the support of the ring $R$, that acts on the triangulated category. 

\begin{corollary}\label{cor:stratified_level}
Let $R$ be a noetherian graded-commutative ring and $\cat{T}$ an $R$-linear triangulated category. We assume
\begin{enumerate}
\item $\cat{T}$ is small and stratified by $R$ in the sense of \cite[Definition~7.1]{Benson/Iyengar/Krause:2015} and $C,M \in \cat{T}$; or
\item $\cat{T}$ is compactly generated and $\Hom[*]{\cat{T}}{C}{C}$ is finitely generated over $R$ for each compact object $C$ and $\cat{T}$ is stratified by $R$ in the sense of \cite[4.2]{Benson/Iyengar/Krause:2011b} and $C,M \in \cat{T}$ compact.
\end{enumerate}
If there exists a $(C,C)$-regular sequence $x_1, \ldots, x_t$ in $R$, then
\begin{equation*}
\supp_R(M) \subseteq \supp_R(C) \cap \cV(x_1, \ldots, x_t) \implies \level_\cat{T}^C(M) \geq t+1
\end{equation*}
for any $M \neq 0$.
\end{corollary}
\begin{proof}
This follows from \cref{level_lower_bound_reg} using \cite[Theorem~7.4]{Benson/Iyengar/Krause:2015} or \cite[Theorem~1.2]{Benson/Iyengar/Krause:2011b}, respectively. 
\end{proof}

\begin{example} For a field $k$ of characteristic $p>0$ and $G$ an elementary abelian $p$-group of rank $r$, consider the homotopy category of injective $kG$-modules $\cat{T}=\kcat{\Inj{kG}}$. Its subcategory of compact objects is equvialent to the bounded derived category $\dbcat{\mod{kG}}$. The category $\cat{T}$ is stratified by the group cohomology ring $R=\ch{}{G,k}$ in the sense of \cite[4.2]{Benson/Iyengar/Krause:2011b} by \cite[Theorem~8.1]{Benson/Iyengar/Krause:2011a} and \cref{cor:stratified_level} applies. For $C=k$, $M\in \dbcat{\mod{kG}}$ and the $r$ polynomial generators $x_1,\ldots, x_r$ of $\ch[*]{G}{k}$, we have 
\begin{equation*}\supp_R(M) \subseteq \supp_R(C) \cap \cV(x_1, \ldots, x_t)
\end{equation*}
if and only if $M$ is a perfect complex. For a perfect complex $M$ it is well-known that $\level_\cat{T}^k(M)$ is a lower bound for the total Loewy length of its homology, thus
\begin{equation*}
    \sum_i \loewy{kG}{H_i(M)} \geq r+1
\end{equation*}
if $H_*(M)\neq 0$. For $p=2$ this is \cite[(II) Theorem~4]{Carlsson:1983} and for odd $p$ see \cite[Section~4.1]{Allday/Puppe:1993}. Alternatively, we can consider $B=kG$ as a commutative noetherian local ring and apply \cref{level_ring_csupp} as in the following section.
\end{example}

%%%%%%%%%%%%%%%%%%%%%%%%%%%%%%%%%%%%%%%%%%%%%%%%%%
\subsection{Action of Hochschild cohomology} \label{level_HH_recover}

\Cref{level_lower_bound_reg} recovers results about lower bounds of level, and Loewy length, from \cite{Avramov/Buchweitz/Iyengar/Miller:2010,Briggs/Grifo/Pollitz:2024}. 

Let $A$ be a commutative ring and $E \colonequals \Kos^A(f_1, \ldots, f_c)$ the Koszul complex on a sequence $f_1, \ldots, f_c$ in $A$ as in \cref{HH_action}. Then $R \colonequals A[\chi_1, \ldots, \chi_c]$ acts on the derived category $\dcat{E}$ through the ring homomorphism $R \to \HH{E/A}$. 

The \emph{cohomological support} of dg $E$-modules $M$ and $N$ is
\begin{equation*}
\csupp_E(M,N) \colonequals \supp_R(\Hom[*]{\dcat{E}}{M}{N})\,;
\end{equation*}
see \cite[4.6]{Liu/Pollitz:2025}. This definition is not the same as in \cite[Definition~5.1.1]{Pollitz:2021}, as the latter considers the support without the irrelevant ideal. 

\begin{lemma} \label{csupp_reg_seq}
Let $A$ be a commutative ring and $E$ the Koszul complex on a sequence $f_1, \ldots, f_c$. Let $M$ be a dg $E$-module and let $x_1, \ldots, x_t$ be an $(M,M)$-regular sequence in $R$. Then
\begin{equation*}
\csupp_E(\kosobj{M}{(x_1, \ldots, x_t)},\kosobj{M}{(x_1, \ldots, x_t)}) = \csupp_E(M,M) \cap \Spec(R/\braket{x_1, \ldots, x_t})\,.
\end{equation*}
\end{lemma}
\begin{proof}
For convenience we write $\bmx = x_1, \ldots, x_t$. Since $\kosobj{M}{\bmx}$ lies in the thick subcategory generated by $M$, we obtain
\begin{equation*}
\csupp_E(\kosobj{M}{\bmx},\kosobj{M}{\bmx}) = \csupp_E(M,\kosobj{M}{\bmx})\,.
\end{equation*}
As $\bmx$ is $(M,M)$-regular, using \cref{Hom_Koszul_object_in_different_variables} we obtain
\begin{equation*}
\begin{aligned}
\csupp_E(M,\kosobj{M}{\bmx}) &= \supp_R(\Hom[*]{\dcat{E}}{M}{\kosobj{M}{\bmx}}) \\
&= \supp_R(\Hom[*]{\dcat{E}}{M}{M}/\braket{\bmx} \Hom[*]{\dcat{E}}{M}{M}) \\
&= \csupp_E(M,M) \cap \Spec(R/\braket{\bmx})\,.
\end{aligned}
\end{equation*}
This finishes the proof.
\end{proof}

\begin{proposition} \label{level_koscx_csupp}
Let $A$ be a regular local ring with residue field $k$ and $f_1, \ldots, f_c$ a sequence in the maximal ideal of $A$. For the Koszul complex $E$ on $f_1, \ldots, f_c$ and a dg $E$-module $M \neq 0$ one has
\begin{equation*}
\level_{\dcat{E}}^k(M) \geq c - \dim \csupp_E(M,M) + 1\,.
\end{equation*}
\end{proposition}
\begin{proof}
We use the notation established above and set $S \colonequals k[\chi_1, \ldots, \chi_c]$. 

We assume $M \in \thick_{\dcat{E}}(k)$, as otherwise there is nothing to show. We set $\fa \colonequals \ann_R \Hom[*]{\dcat{E}}{M}{M}$. Let $x_1, \ldots, x_t \in \fa$ be a maximal $S$-regular sequence in $\fa$. Then $x_1, \ldots, x_t$ is $(k,k)$-regular and $\braket{x_1, \ldots, x_t} \subseteq \fa$. Hence
\begin{equation*}
\begin{aligned}
\csupp_E(M,M) &= \Spec(R/\fa) = \csupp_E(k,k) \cap \Spec(R/\fa) \\
&\subseteq \csupp_E(k,k) \cap \Spec(R/\braket{x_1, \ldots, x_t}) \\
&= \csupp_E(\kosobj{k}{(x_1, \ldots, x_t)},\kosobj{k}{(x_1, \ldots, x_t)})\,.
\end{aligned}
\end{equation*}
As $A$ is regular and $M$ has finitely generated total homology, by \cite[Proposition~4.8]{Liu/Pollitz:2025}, we have $M \in \thick_{\dcat{E}}(\kosobj{k}{(x_1, \ldots, x_t)})$. Now we can apply \cref{level_lower_bound_reg} to obtain
\begin{equation*}
\level_{\dcat{E}}^k(M) \geq t + 1 = \fa\text{-}\depth_R(S) + 1\,.
\end{equation*}
As $S$ is a Cohen--Macaulay module, one has
\begin{equation*}
\fa\text{-}\depth_R(S) = \dim(S) - \dim(S \otimes_R R/\fa) + 1\,;
\end{equation*}
this holds by the graded version of \cite[Theorem~2.12]{Bruns/Herzog:1998}. It remains to observe that
\begin{equation*}
\dim(S) = c \quad \text{and} \quad \dim(S \otimes_R R/\fa) = \dim(\csupp_E(M,M))\,. \qedhere
\end{equation*}
\end{proof}

For the Koszul complex $E$ there is a dg algebra morphism 
\begin{equation*}
E \to \hh[0]{}{E} = A/\braket{f_1, \ldots, f_c} \equalscolon B\,.
\end{equation*}
This morphism induced an exact functor $\dcat{B} \to \dcat{E}$ via restriction. Hence the lower bound for level in $\dcat{E}$ can be used to obtain a similar bound for level in $\dcat{B}$. To express the bound purely in terms of $B$, we recall the definition of the cohomological support for local rings from \cite[Section~6.1]{Pollitz:2021}.

Let $B$ be a commutative local ring and $\hat{B}$ its completion with respect to the maximal ideal. The \emph{cohomological support} of $M,N \in \dcat{B}$ is
\begin{equation*}
\csupp_B(M,N) \colonequals \csupp_E(\hat{B} \otimes_B M, \hat{B} \otimes_B N)
\end{equation*}
where $E$ is a minimal derived complete intersection approximation of $B$. That is, by the Cohen Structure Theorem there exists a complete local ring $A$ with maximal ideal $\fm$ and an ideal $\fa \subseteq \fm^2$ such that $\hat{B} = A/\fa$. Then $E$ is the Koszul complex on a minimal generating set of $\fa$. By \cite[6.1.3]{Pollitz:2021}, the cohomological support $\csupp_B(M,N)$ is independent of the choice of $A$ and the minimal generating set of $\fa$.

\begin{corollary} \label{level_ring_csupp}
Let $B$ be a commutative noetherian local ring with maximal ideal $\fn$ and residue field $k=B/\fn$. For any $B$-complex $M \neq 0$ one has
\begin{equation*}
\level_{\dcat{B}}^k(M) \geq \codim(B) - \dim \csupp_B(M,M) + 1\,,
\end{equation*}
where $\codim(B) = \rank_k(\fn/\fn^2) - \dim(B)$ is the \emph{codimension of $B$}. 
\end{corollary}
\begin{proof}
We assume $M \in \thick_{\dcat{B}}(k)$, as otherwise there is nothing to show. Then $M \in \dbcat{\mod{B}}$ and $\hat{M} \cong \hat{B} \otimes_B M$. Hence
\begin{equation*}
\begin{aligned}
\level_{\dbcat{B}}^k(M) &= \level_{\dcat{\hat{B}}}^k(\hat{M}) \geq \level_{\dcat{E}}^k(\hat{M}) \\
&\geq \codim(B) - \dim \csupp_E(\hat{M},\hat{M}) + 1 \\
&= \codim(B) - \dim \csupp_B(M,M) + 1\,,
\end{aligned}
\end{equation*}
using, in this order, \cite[Corollary~2.12]{Letz:2021}, \cite[Lemma~2.4(6)]{Avramov/Buchweitz/Iyengar/Miller:2010} and \cref{level_koscx_csupp}. 
\end{proof}

When $M$ is a perfect $B$-complex, then $\dim \csupp_B(M,M) = 0$. Hence \cref{level_ring_csupp} recovers \cite[Theorem~3]{Avramov/Buchweitz/Iyengar/Miller:2010} and \cite[Theorem~C, Theorem~3.1]{Briggs/Grifo/Pollitz:2024}.

\begin{remark}
The cohomological support for a complete intersection can also be obtained as the support of a triangulated category as in \cref{support}; see \cite[Section~11]{Benson/Iyengar/Krause:2008} for the details.
\end{remark}

%%%%%%%%%%%%%%%%%%%%%%%%%%%%%%%%%%%%%%%%%%%%%%%%%%
\subsection{Level with respect to a Koszul object}

We provide another bound for level using a regular element. In spirit it is opposite to \cref{level_lower_bound_reg}, as the generator is the Koszul object with respect to the regular sequence.

\begin{lemma} \label{level_kosobj_kosobj_lower}
Let $R$ be a graded-commutative ring and $\cat{T}$ an $R$-linear triangulated category. We fix $x \in R$ and $M \in \cat{T}$.
We assume that $x$ is $M$-productive. If $x$ is an $(M,M)$-regular sequence, then
\begin{equation*}
\level^{\kosobj{M}{x^m}}(\kosobj{M}{x^n}) \geq \left\lceil \frac{n}{m} \right\rceil
\end{equation*}
for any integers $0 \leq m \leq n$.
\end{lemma}
\begin{proof}
We first show that $x^m(\kosobj{M}{x^n})$ is $\kosobj{M}{x^m}$-ghost whenever $m < n$. Given any morphism $f \colon \susp^d \kosobj{M}{x^m} \to \kosobj{M}{x^n}$ we obtain
\begin{equation*}
x^m(\kosobj{M}{x^n}) \circ f = (\susp^{m|x|} f) \circ (\susp^d x^m(\kosobj{M}{x^m})) = 0
\end{equation*}
as $x^m$ is $M$-productive by assumption and \cref{productive_product}. Hence $x^m(\kosobj{M}{x^n})$ is $\kosobj{M}{x^m}$-ghost. 

It remains to show that $x^{n-1}(\kosobj{M}{x^n}) \neq 0$. We assume $x^{n-1}(\kosobj{M}{x^n}) = 0$. Considering the commutative diagram
\begin{equation*}
\begin{tikzcd}
& \susp^{|x|} M \ar[dl,dashed,"f"] \ar[d,"x^{n-1}" description] \ar[dr,"0"] \\
M \ar[r,"x^n"] & \susp^{n|x|} M \ar[r] & \kosobj{M}{x^n} \ar[r] & \susp M
\end{tikzcd}
\end{equation*}
where the bottom row is an exact triangle, we obtain a morphism $f \colon \susp^{|x|} M \to M$ such that $x^{n-1} (xf - \id) = 0$. As $x$ is a non-zero divisor on $\Hom[*]{\cat{T}}{M}{M}$, we obtain $xf = \id$. By the proof of \cref{zero_one} this yields $\kosobj{M}{x} = 0$. This is a contradiction as $x$ is $(M,M)$-regular sequence.

Therefore, for any integer $t$, with $mt\leq n$, the morphism $x^{m(t-1)}(\kosobj{M}{x^n})$ is non-zero and $(t-1)$-fold $\kosobj{M}{x^m}$-ghost.
\end{proof}

In the case $m = 1$ we immediately obtain an equality.

\begin{corollary} \label{level_kosobj_kosobj_x}
Let $R$ be a graded-commutative ring and $\cat{T}$ an $R$-linear triangulated category. Let $x \in R$ and $M \in \cat{T}$.
We assume that $x$ is $M$-productive. If $x$ is an $(M,M)$-regular sequence, then
\begin{equation*}
\level^{\kosobj{M}{x}}(\kosobj{M}{x^n}) = n
\end{equation*}
for any positive integer $n$.
\end{corollary}
\begin{proof}
The inequality $\geq$ holds by \cref{level_kosobj_kosobj_lower}, and the reverse inequality follows from the exact triangle in \cref{kos_product_triangle}.
\end{proof}

The exact triangle in \cref{kos_product_triangle} is not enough for the lower bound when $n$ is not a multiple of $m$. In this situation we use a higher octahedral axiom, which holds in an enhanced triangulated category; see \cref{app:higher_octa}. 

\begin{proposition} \label{level_kosobj_kosobj_equal}
Let $R$ be a graded-commutative ring and $\cat{T}$ an $R$-linear topological triangulated category. Let $x \in R$ and $M \in \cat{T}$. We assume that $x$ is $M$-productive. If $x$ is an $(M,M)$-regular sequence, then
\begin{equation*}
\level^{\kosobj{M}{x^m}}(\kosobj{M}{x^n}) = \left\lceil \frac{n}{m} \right\rceil
\end{equation*}
for positive integers $m \leq n$. 
\end{proposition}
\begin{proof}
The inequality $\geq$ holds by \cref{level_kosobj_kosobj_lower}. For the reverse inequality we construct exact triangles. We write $n = km + \ell$ for $k \geq 1$ and $0 \leq \ell < m$. Then there are exact triangles
\begin{equation*}
\kosobj{M}{x^m} \to \kosobj{M}{x^{im}} \to \susp^{m|x|} \kosobj{M}{x^{(i-1)m}} \to \susp \kosobj{M}{x^m}
\end{equation*}
for $1 \leq i \leq k$. by \cref{kos_product_triangle}. This shows the inequality for $\ell = 0$. When $\ell \neq 0$ we apply the higher octahedral axiom to the composition
\begin{equation*}
M \xrightarrow{x^{\ell}} \susp^{\ell|x|} M \xrightarrow{x^{m-\ell}} \susp^{m|x|} M \xrightarrow{x^{(k-1)m+\ell}} \susp^{n|x|} M\,;
\end{equation*}
see \cref{octahedral_iterated}. Then there exists a homotopy cartesian square
\begin{equation*}
\begin{tikzcd}
\kosobj{M}{x^m} \ar[r] \ar[d] & \susp^{\ell|x|} \kosobj{M}{x^{m-\ell}} \ar[d] \\
\kosobj{M}{x^{n}} \ar[r] & \susp^{\ell|x|} \kosobj{M}{x^{km}} \nospacepunct{,}
\end{tikzcd}
\end{equation*}
and hence an exact triangle
\begin{equation*}
\kosobj{M}{x^m} \to \kosobj{M}{x^n} \oplus \susp^{\ell|x|} \kosobj{M}{x^{m-\ell}} \to \susp^{\ell|x|} \kosobj{M}{x^{km}} \to \susp \kosobj{M}{x^m}\,.
\end{equation*}
This establishes $\leq$.
\end{proof}

%%%%%%%%%%%%%%%%%%%%%%%%%%%%%%%%%%%%%%%%%%%%%%%%%%
%%%%%%%%%%%%%%%%%%%%%%%%%%%%%%%%%%%%%%%%%%%%%%%%%%
\appendix

%%%%%%%%%%%%%%%%%%%%%%%%%%%%%%%%%%%%%%%%%%%%%%%%%%
%%%%%%%%%%%%%%%%%%%%%%%%%%%%%%%%%%%%%%%%%%%%%%%%%%
\section{Hom-sets involving Koszul objects}

Let $R$ be a graded-commutative ring, $x\in R$, and $\cat{T}$ an $R$-linear triangulated category with objects $C$, $M$. In \cref{induced_iso_reg_elt} we showed there is an $R$-linear isomorphism
\begin{equation*}
\Hom[*]{\cat T}{\susp^{-1} \kosobj{C}{x}}{M} \cong \Hom[*]{\cat T}{C}{\susp^{-|x|} \kosobj{M}{x}}
\end{equation*}
provided $x$ is $(C,M)$-regular. In \cite[(3.1)]{Benson/Iyengar/Krause/Pevtsova:2021} it is stated, without proof, that there is an isomorphism without assuming that $x$ is $(C,M)$-regular; compare \cref{Hom_Kos_BIKP}. We provide an example in which there is no $R$-linear isomorphism.

\begin{example}\label{Counterexample:Hom_Kos_different_variable}
Consider the ring of dual numbers $A=k[x]/(x^2)$ over a field $k$. The center of $\cat{T}= \dbcat{\mod{A}}$ is the trivial extension ring $k[\zeta] \ltimes \prod_{r \geqslant 0} k$ with $\zeta$ acting trivially on the product; see \cite[Proposition~5.4]{Krause/Ye:2011}. The element $\zeta$ is of degree $2$ if the characteristic of $k$ is not $2$ and of degree $1$ otherwise.

The category $\dbcat{\mod{A}}$ is a Krull-Schmidt category and the indecomposable objects are the chain complexes $A_m^n$ of the form
\begin{equation*}
\ldots \rightarrow 0\rightarrow A\xrightarrow{x} A\xrightarrow{x}\ldots \xrightarrow{x} A\rightarrow 0\rightarrow\ldots
\end{equation*}
for $n\geq m$ with $0$ in degrees $>n$ and in degrees $<m$. We write $x^n_m$ for the endomorphism of $A^n_m$ given by multiplication with $x$ in degree $m$ and the zero homomorphism otherwise. As in \cite[Example 4.7]{Letz/Stephan:2025}, we consider $\eta_0, \eta_1 \in \prod_{r \geqslant 0} k$ determined by
\begin{equation*}
\eta_r(A_m^n) = \begin{cases} 
x_m^n & n-m = r \\
0 & n-m \neq r \nospacepunct{.}
\end{cases}
\end{equation*}
For the element $\eta = \eta_0+\eta_1$ of degree zero, we obtain $\kosobj{A_n^n}{\eta}\cong A^{n+1}_n$ and $\kosobj{A^n_{n-1}}{\eta}\cong A^{n+1}_{n-1} \oplus A^n_n$. Thus
\begin{equation*}
\Hom{\cat T}{\susp^{-1} \kosobj{A_n^n}{\eta}}{A^n_{n-1}} \cong \Hom{\cat T}{A^n_{n-1}}{A^n_{n-1}}\cong k\oplus k
\end{equation*}
as $A$-modules which is not isomorphic to the $A$-module
\begin{equation*}
\Hom{\cat T}{A^n_n}{\kosobj{A^n_{n-1}}{\eta}} \cong \hh[n]{}{A^{n+1}_{n-1}\oplus A^n_n} \cong A\nospacepunct{.}
\end{equation*}
\end{example}

%%%%%%%%%%%%%%%%%%%%%%%%%%%%%%%%%%%%%%%%%%%%%%%%%%
%%%%%%%%%%%%%%%%%%%%%%%%%%%%%%%%%%%%%%%%%%%%%%%%%%
\section{Higher octahedral axiom} \label{app:higher_octa}

It is well-known that the notion of a triangulated category is sometimes insufficient. The triangulated categories that appear in algebra and topology are equipped with further structure. For \cref{level_kosobj_kosobj_equal} we require a higher version of the octahedral axiom applied to an $n$-fold composition for $n \geq 2$. A similar property previously appeared in \cite{Maltsiniotis:2006} as $n$-exact triangles; also see \cite[\S13]{Groth/Stovicek:2016}. We provide a higher octahedral axiom for topological triangulated categories. A \emph{topological} triangulated category is the homotopy category of a stable cofibration category; see \cite{Schwede:2013}. 

\begin{proposition} \label{octahedral_iterated}
Let $\cat{T}$ be a topological triangulated category and $n$ a positive integer. Given a sequence of morphisms
\begin{equation*}
X_0 \xrightarrow{f_1} X_1 \xrightarrow{f_2} X_2 \to \cdots \to X_{n-1} \xrightarrow{f_n} X_n
\end{equation*}
in $\cat{T}$ we set $f_{i,j} \colonequals f_j \circ \cdots \circ f_i$ for $i \leq j$. Let 
\begin{equation*}
X_{i-1} \xrightarrow{f_{i,j}} X_j \xrightarrow{g_{i,j}} \cone(f_{i,j}) \xrightarrow{h_{i,j}} \susp X_{i-1}
\end{equation*}
be exact triangles in $\cat{T}$ for $1 \leq i \leq j \leq n$. Then there exist morphisms
\begin{equation*}
\begin{aligned}
u_{i,j} &\colon \cone(f_{i,j}) \to \cone(f_{i+1,j}) \quad \text{for } 1 \leq i < j \leq n \quad \text{and} \\
v_{i,j} &\colon \cone(f_{i,j}) \to \cone(f_{i,j+1}) \quad \text{for } 1 \leq i \leq j < n
\end{aligned}
\end{equation*}
such that
\begin{enumerate}
\item \label{octahedral_iterated:triangles} the triangles
\begin{equation*}
\cone(f_{i,j}) \xrightarrow{v_{i,k-1} \circ \cdots \circ v_{i,j}} \cone(f_{i,k}) \xrightarrow{u_{j,k} \circ \cdots \circ u_{i,k}} \cone(f_{j+1,k}) \xrightarrow{(\susp g_{i,j}) \circ h_{j+1,k}}
\end{equation*}
are exact for $1 \leq i \leq j < k \leq n$;
\item \label{octahedral_iterated:comp_three} the following diagrams commute
\begin{equation*}
\begin{tikzcd}
X_{j+1} \ar[r,"g_{i,j+1}"] \ar[dr,"g_{i+1,j+1}" swap] & \cone(f_{i,j+1}) \ar[d,"u_{i,j+1}"] \\
& \cone(f_{i+1,j+1})
\end{tikzcd}
\quad \text{and} \quad
\begin{tikzcd}
\cone(f_{i,j}) \ar[d,"v_{i,j}" swap] \ar[dr,"h_{i,j}"] &  \\
\cone(f_{i,j+1}) \ar[r,"h_{i,j+1}" swap] & \susp X_{i-1}
\end{tikzcd}
\end{equation*}
for $1 \leq i \leq j < n$; 
\item \label{octahedral_iterated:comp_four} the following squares commute and are homotopy cartesian
\begin{equation*}
\begin{tikzcd}
X_j \ar[r,"f_{j+1}"] \ar[d,"g_{i,j}"] & X_{j+1} \ar[d,"g_{i,j+1}"] \\
\cone(f_{i,j}) \ar[r,"v_{i,j}"] & \cone(f_{i,j+1})
\end{tikzcd}
\quad \text{and} \quad
\begin{tikzcd}
\cone(f_{i,j+1}) \ar[r,"u_{i,j+1}"] \ar[d,"h_{i,j+1}"] & \cone(f_{i+1,j+1}) \ar[d,"h_{i+1,j+1}"] \\
\susp X_{i-1} \ar[r,"\susp f_i"] & \susp X_i
\end{tikzcd} 
\end{equation*}
for $1 \leq i \leq j < n$;
\item \label{octahedral_iterated:extra} the following square commutes and is homotopy cartesian
\begin{equation*}
\begin{tikzcd}
\cone(f_{i-1,i}) \ar[d,"u_{i-1,i}"] \ar[r,"v_{i-1,i}"] & \cone(f_{i-1,i+1}) \ar[d,"u_{i-1,i+1}"] \\
\cone(f_{i,i}) \ar[r,"v_{i,i}"] & \cone(f_{i,i+1})
\end{tikzcd}
\end{equation*}
for $1 < i < n$.
\end{enumerate}
\end{proposition}

For $n = 2$ condition \cref{octahedral_iterated:triangles,octahedral_iterated:comp_three} and the first part of \cref{octahedral_iterated:comp_four} is the classical octahedral axiom, while the second part of \cref{octahedral_iterated:comp_four} is the additional condition for strong triangulated categories; see \cite[Definition~3.8]{May:2001}. For $n > 2$ the compatibilities \cref{octahedral_iterated:comp_three,octahedral_iterated:comp_four} are obtained by applying the classical octahedral axiom to the compositions $f_{j+1,k} \circ f_{i,j}$ for $1 \leq i \leq j < k \leq n$. Additionally, \cref{octahedral_iterated:triangles,octahedral_iterated:extra} and the second part of \cref{octahedral_iterated:comp_four} are additional conditions.

\begin{proof}
Let $\cat{C}$ be a stable cofibration category with $\cat{T} = \Ho(\cat{C})$ and let $\gamma \colon \cat{C} \to \cat{T}$ be the localization functor. By \cite[Theorem~A.1]{Schwede:2013} any morphism in $\cat{T}$ is of the form $\gamma(s)^{-1} \gamma(f')$ where $s$ is an acyclic cofibration and $f$ a morphism in $\cat{C}$. Further, in $\cat{C}$ every morphism factors as a cofibration followed by a weak equivalence. So we obtain a commutative diagram
\begin{equation*}
\begin{tikzcd}[column sep=large]
\hat{X}_0 \ar[r,"\gamma(\hat{f}_1)"] \ar[d,"="] & \hat{X}_1 \ar[r,"\gamma(\hat{f}_2)"] \ar[d,"\gamma(t_1)","\cong" swap] & \hat{X}_2 \ar[r] \ar[d,"\gamma(t_2)","\cong" swap] & \cdots \ar[r] & \hat{X}_{n-1} \ar[r,"\gamma(\hat{f}_n)"] \ar[d,"\gamma(t_{n-1})","\cong" swap] & \hat{X}_n \ar[d,"\gamma(t_n)","\cong" swap] \\
\bar{X}_0 \ar[r,"\gamma(\bar{f}_1)"] \ar[d,"="] & \bar{X}_1 \ar[r,"\gamma(\bar{f}_2)"] \ar[d,"\gamma(s_1)^{-1}","\cong" swap] & \bar{X}_2 \ar[r] \ar[d,"\gamma(s_2)^{-1}","\cong" swap] & \cdots \ar[r] & \bar{X}_{n-1} \ar[r,"\gamma(\bar{f}_n)"] \ar[d,"\gamma(s_{n-1})^{-1}","\cong" swap] & \bar{X}_n \ar[d,"\gamma(s_n)^{-1}","\cong" swap] \\
X_0 \ar[r,"f_1"] & X_1 \ar[r,"f_2"] & X_2 \ar[r] & \cdots \ar[r] & X_{n-1} \ar[r,"f_n"] & X_n
\end{tikzcd}
\end{equation*}
where each vertical arrow is an isomorphism in $\cat{T}$ and each $\hat{f}_i$ is a cofibration is $\cat{C}$. Hence it is enough to show the claim for $f_i = \gamma(\hat{f_i})$. 

We take iterated pushouts in $\cat{C}$ to obtain
\begin{equation*}
\begin{tikzcd}
X_0 \ar[r,tail,"\hat{f}_1"] \ar[d] & X_1 \ar[r,tail,"\hat{f}_2"] \ar[d,"\hat{g}_{1,1}"] & X_2 \ar[r,tail,"\hat{f}_3"] \ar[d,"\hat{g}_{1,2}"] & X_3 \ar[r,tail] \ar[d,"\hat{g}_{1,3}"] & \cdots \\
\ast \ar[r,tail] & X_1/X_0 \ar[r,tail,"\hat{v}_{1,1}"] \ar[d] & X_2/X_0 \ar[r,tail,"\hat{v}_{1,2}"] \ar[d,"\hat{u}_{1,2}"] & X_3/X_0 \ar[r,tail] \ar[d,"\hat{u}_{1,3}"] & \cdots \\
& \ast \ar[r,tail] & X_2/X_1 \ar[r,tail,"\hat{v}_{2,2}"] \ar[d] & X_3/X_1 \ar[r,tail] \ar[d,"\hat{u}_{2,3}"] & \cdots \\
&& \ast \ar[r,tail] & X_3/X_2 \ar[r,tail] & \cdots \nospacepunct{.}
\end{tikzcd}
\end{equation*}
We set
\begin{equation*}
\hat{g}_{i,j} \colonequals \hat{u}_{i,j} \circ \cdots \circ \hat{u}_{1,j} \circ \hat{g}_{1,j}
\end{equation*}
for $1 \leq i \leq j \leq n$. We may assume
\begin{equation*}
g_{i,j} = \gamma(\hat{g}_{i,j}) \quad \text{and} \quad h_{i,j} = \delta(\hat{f}_j \circ \cdots \circ \hat{f}_i)\,.
\end{equation*}
We further set
\begin{equation*}
u_{i,j} \colonequals \gamma(\hat{u}_{i,j}) \quad \text{and} \quad v_{i,j} \colonequals \gamma(\hat{v}_{i,j})\,.
\end{equation*}

From the pushout squares above and the naturality of $\delta$ by \cite[Proposition~A.11]{Schwede:2013}, we obtain
\begin{equation*}
\delta(\hat{v}_{i,k-1} \circ \cdots \circ \hat{v}_{i,j}) = \delta(\hat{f}_{j,k}) \circ \susp g_{i,j}\,.
\end{equation*}
This yields the exact triangles \cref{octahedral_iterated:triangles}.

The diagrams in \cref{octahedral_iterated:comp_three,octahedral_iterated:comp_four,octahedral_iterated:extra} hold by construction and \cite[Proposition~A.11]{Schwede:2013}. The squares in \cref{octahedral_iterated:comp_four,octahedral_iterated:extra} are homotopy cartesian by \cite[Lemma~5.7, Proposition~5.8]{Letz/Stephan:2025}. 
\end{proof}

\bibliographystyle{amsalpha}
\bibliography{main}

\end{document}